\documentclass[12pt]{iopart}

\newcommand{\la}{\langle} \newcommand{\ra}{\rangle}

\renewcommand{\Re}{\mathrm{Re}\,}
\renewcommand{\Im}{\mathrm{Im}\,}

\newcommand{\ls}{\lesssim} \newcommand{\gs}{\gtrsim}

\newcommand{\R}{\mathbb{R}} \newcommand{\C}{\mathbb{C}}
\newcommand{\N}{\mathbb{N}} \newcommand{\Z}{\mathbb{Z}}
\newcommand{\supp}{\mathrm{supp}\, }

\usepackage{amsthm}
\usepackage{iopams}

\usepackage{graphicx,color}
\usepackage{epsfig}

\newtheorem{theorem}{Theorem}[section]

\newtheorem{proposition}[theorem]{Proposition}

\newtheorem{lemma}[theorem]{Lemma}

\theoremstyle{remark}
\newtheorem{remark}{Remark}

\theoremstyle{definition}
\newtheorem{definition}[theorem]{Definition}

\eqnobysec

\begin{document}

\title[On the 2d Zakharov system]{On the 2d Zakharov system with $L^2$ Schr\"odinger data}

\author{I. Bejenaru$^1$, S. Herr$^2$, J. Holmer$^3$ and D. Tataru$^2$}

\address{$^1$ Department
  of Mathematics, University of Chicago, Chicago, IL 60637, USA}
\ead{\mailto{bejenaru@math.uchicago.edu}}

\address{$^2$ Department of
  Mathematics, University of California, Berkeley, CA 94720-3840, USA}
\eads{\mailto{herr@math.berkeley.edu}, \mailto{tataru@math.berkeley.edu}}

\address{$^3$ Department of
  Mathematics, Brown University, Box 1917, 151 Thayer St., Providence,
  RI 02912, USA} \ead{holmer@math.brown.edu}

\begin{abstract}
  We prove local in time well-posedness for the Zakharov system in two
  space dimensions with large initial data in $L^2\times
  H^{-1/2}\times H^{-3/2}$. This is the space of optimal regularity in
  the sense that the data-to-solution map fails to be smooth at the
  origin for any rougher pair of spaces in the $L^2$-based Sobolev
  scale. Moreover, it is a natural space for the Cauchy problem in
  view of the subsonic limit equation, namely the focusing cubic
  nonlinear Schr\"odinger equation. The existence time we obtain
  depends only upon the corresponding norms of the initial data -- a
  result which is false for the cubic nonlinear Schr\"odinger equation
  in dimension two -- and it is optimal because Glangetas--Merle's
  solutions blow up at that time.
\end{abstract}

\ams{35Q55}

\section{Introduction and main result}\label{sect:intro}
We study the initial-value problem for the Zakharov system in two
spatial dimensions:
\begin{equation}\eqalign{%
      i\partial_t u + \Delta u = nu, \cr
      \partial_t^2 n - \Delta n = \Delta |u|^2,}\label{eq:zs}
  \end{equation}
  where $u:\mathbb{R}^{2+1}\to \mathbb{C}$ and $n:\mathbb{R}^{2+1}\to
\mathbb{R}$, with initial data 
\[
(u|_{t=0}, n|_{t=0}, \partial_t
n|_{t=0}) = (u_0,n_0,n_1).
\]
This system was introduced by Zakharov \cite{zakharov_collapse_1972}
as a model for the propagation of Langmuir waves in a plasma.

We address the question of local well-posedness of \eref{eq:zs} for
large data in low regularity Sobolev spaces.  For $k,\ell \in \R$ we
 define the space 
\[
\mathbf{H}^{k,\ell}:=H^k(\R^2;\C)\times H^{\ell}(\R^2;\R)\times
H^{\ell-1}(\R^2;\R)
\]
with the natural norm.  By $\mathbf{X}^{k,\ell}_T$ we denote the space
of all tempered distributions $(u,n)$ on $(0,T)\times \mathbb{R}^2$
such that
\begin{eqnarray*}
    & u \in C( [0,T]; H^k(\mathbb{R}^2;\C)), \\
    & n \in C( [0,T]; H^\ell(\mathbb{R}^2;\R)) \cap C^1( [0,T];
    H^{\ell-1}(\mathbb{R}^2;\R)).
  \end{eqnarray*}
  with the standard norm, see \eref{eq:x-norm}. For $0<r\leq
R$ we also define
\begin{equation*}
  \mathbf{H}^{k,\ell}_{R,r}:=\{ (u_0,n_0,n_1) \in \mathbf{H}^{k,\ell}:\
  \|(u_0,n_0,n_1)\|_{\mathbf{H}^{k,\ell}}\leq R; \|u_0\|_{L^2}\leq r\}
\end{equation*}
as a metric subspace of $\mathbf{H}^{k,\ell}$.

Our main result is the local well-posedness of \eref{eq:zs} in
$\mathbf{H}^{0,-\frac12}$, which was phrased as an
open problem by Merle \cite[p. 58, ll. 14--15]{merle_blowup_1998}.
\begin{theorem}\label{main}
  For every $0<r\leq R$ and initial data $(u_0,n_0,n_1)\in
  \mathbf{H}^{0,-\frac12}_{R,r}$ and time $T \ls \min\{\la R\ra^{-2}r^{-2},1\}$,
  there exists a subspace $\mathbf{X}_T\subset
  \mathbf{X}_T^{0,-\frac12}$ and a unique solution $(u,n)\in \mathbf{X}_T$ of the Cauchy
  problem \eref{eq:zs}. The map
  \begin{equation*}
    \mathbf{H}_{R,r}^{0,-\frac12}
    \longrightarrow \mathbf{X}_T^{0,-\frac12}:
    \quad (u_0,n_0,n_1) \mapsto (u,n)
  \end{equation*}
  is locally Lipschitz-continuous.
\end{theorem}

\begin{remark}
Note that a-priori the nonlinear system \eref{eq:zs} is not well-defined for rough distributions.
The precise notion of \emph{solution} in Theorem \ref{main} is explained in Section \ref{sect:reduced}.
The auxiliary space $\mathbf{X}_T$ is based on generalized Fourier restriction spaces.
\end{remark}

\begin{remark}
Notice that Theorem \ref{main} implies in particular that locally the flow map for smooth data
extends continuously to initial data in $\mathbf{H}^{0,-\frac12}$.  The uniqueness claim in Theorem \ref{main} is restricted to the subspace $\mathbf{X}_T$ of $X_T^{0,-\frac12}$, which ensures that $(u,n)$ is the unique limit of smooth solutions.
\end{remark}

Local well-posedness of \eref{eq:zs} in the low-regularity setting
has been previously considered by many authors: Bourgain--Colliander
\cite{bourgain_wellposedness_1996} proved local well-posedness in spaces which comprise the energy space
and established global well-posedness in the energy space under a smallness condition. The local result has been improved later by Ginibre--Tsutsumi--Velo \cite{ginibre_cauchy_1997}. Both aforementioned approaches are based on the Fourier
restriction norm method. For previous well-posedness results
we refer the reader to the references in \cite{bourgain_wellposedness_1996,ginibre_cauchy_1997}. In
\cite{ginibre_cauchy_1997} Ginibre--Tsutsumi--Velo obtain local well-posedness of
\eref{eq:zs} in the case of space dimension $d=2$ in the space $H^k\times H^\ell\times
H^{\ell-1}$ for $(k,\ell)$ confined to the strip $\ell \geq 0$,
$2k\geq \ell +1$. The optimal corner of this strip occurs at
$H^\frac12\times L^2\times H^{-1}$, one-half a derivative away from
the result in Theorem \ref{main}.

One motivation for considering the space $L^2\times H^{-1/2}\times
H^{-3/2}$ is the connection to the cubic nonlinear Schr\"odinger
equation in two spatial dimensions
\begin{equation}\label{eq:nls}
 i\partial_t u + \Delta u + |u|^2u=0.
\end{equation}
Consider the Zakharov system with wave speed $\lambda>0$:
\begin{equation}\eqalign{
     i\partial_t u + \Delta u = nu, \cr
     \frac{1}{\lambda^2}\partial_t^2 n - \Delta n = \Delta |u|^2.}\label{eq:zs-speed}
 \end{equation}
 
Then formally \eref{eq:zs-speed} converges to \eref{eq:nls} as $\lambda \to
\infty$ in the sense that for fixed initial data $u_\lambda \to u$,
where $(u_\lambda, n_\lambda)$ solves \eref{eq:zs-speed} and $u$
solves \eref{eq:nls} with the same initial data.  Rigorous results
of this type in a high regularity setting were obtained by Schochet--Weinstein \cite{schochet_nonlinear_1986},
Added--Added \cite{added_equations_1988}, Ozawa--Tsutsumi
\cite{ozawa_nonlinear_1992}, see also the recent work by Masmoudi--Nakanishi \cite{masmoudi_energy_2008} on this issue in 3d.

Local well-posedness in $L^2$ of \eref{eq:nls} was obtained by
Cazenave--Weissler \cite{cazenave_cauchy_1990}.  However, in this
version of well-posedness, the time interval of existence depends
directly upon the initial data, not just on the $L^2$ norm of the
initial data. Indeed, via the pseudoconformal transformation, it can
be shown that a result giving the maximal time of existence in terms
of the $L^2$ norm alone is not possible\footnote{Note that
  Killip-Tao-Visan \cite{killip_cubic_2008} have recently obtained
  global well-posedness for \eref{eq:nls} if $u_0\in L^2$ is radial
  and $\|u_0\|_{L^2}<\|Q\|_{L^2}$, see \eref{eq:gs}.}.

\begin{remark}
  Our result gives local well-posedness of \eref{eq:zs-speed} with a time of existence
  depending on the $L^2$ norm of $u_0$, but also on the $H^{-1/2}\times
  H^{-3/2}$ norm of the wave data $(n_0,n_1)$ as well as the wave
  speed. Indeed, this claim follows by combining the rescaling
  \[
  u_{\lambda}(t,x)=\lambda u(\lambda^2 t,\lambda x), \quad   v_{\lambda}(t,x)=\lambda^2 v(\lambda^2 t, \lambda x)
  \]
  and Theorem \ref{main}. However, note that the lower bound on the maximal
  time of existence obtain by this method tends to zero as the wave speed goes to
  infinity.
\end{remark}

Global well-posedness of \eref{eq:zs} is known for initial data in
the energy space $H^1\times L^2\times H^{-1}$ with $\|u_0\|_{L^2}\leq
\|Q\|_{L^2}$, see \cite{bourgain_wellposedness_1996,
  glangetas_concentration_1994}; see also \cite{colliander_regularity_2002} regarding bounds on higher order Sobolev norms. Recently, the imposed regularity
assumption has been slightly weakened in \cite{fang_low_2008}. Here,
$Q$ is the ground state solution for \eref{eq:nls},
i.e. $Q$ is the unique  solution to
\begin{equation}\label{eq:gs}
  -Q+\Delta Q+ |Q|^2Q=0, \; Q>0,\; Q(x)=Q(|x|), \; Q\in \mathcal{S}(\R^2)
\end{equation}
of minimal $L^2$ mass. This gives rise to a blow-up solution of
\eref{eq:nls} by the pseudoconformal transformation.  This idea is
exploited in \cite{glangetas_existence_1994}, where Glangetas--Merle
construct a family of blow-up solutions for \eref{eq:zs} of the form
\begin{equation}
  \eqalign{
    u(t,x)=\frac{\omega}{T-t}e^{i\left(\theta+\frac{\omega^2}{T-t}-\frac{|x|^2}{4(T-t)}\right)}
    P_\omega\left(\frac{x\omega}{T-t}\right)\cr
    n(t,x)=\left(\frac{\omega}{T-t}\right)^2 N_\omega\left(\frac{x
        \omega}{T-t}\right)}\label{eq:blowup_sol}
\end{equation}
for parameters $\theta \in \mathbb{S}^1$, $T>0$, and $\omega\gg 1$,
such that $P_{\omega}\in H^1$ is smooth and radially symmetric, $N_{\omega}\in L^2$ is a radially symmetric
Schwartz function, and $(P_{\omega},N_{\omega})\to (Q,-Q^2)$ in
$H^1\times L^2$ as $\omega \to \infty$. In particular, this implies
the necessity of the smallness assumption $\|u_0\|_{L^2}\leq
\|Q\|_{L^2}$ for any global existence result for \eref{eq:zs}.

We prove Theorem \ref{main} by the contraction method in a suitably
defined Fourier restriction norm space, which gives a certain lower
bound on the time of existence. By adapting the
argument of Colliander--Holmer--Tzirakis \cite{colliander_low_2008}
using the $L^2$ conservation of $u(t)$ and iteration, we are able to
show that this time can in fact be extended to the longer lifespan given in
Theorem \ref{main}. In summary, the time of existence we obtain is
based on
\begin{enumerate}
\item sharp multilinear estimates
\item the $L^2$ conservation law for the Schr\"odinger part.
\end{enumerate}
Reviewing the solutions \eref{eq:blowup_sol} constructed by
Glangetas--Merle we observe that Theorem \ref{main} contains the
optimal\footnote{up to the implicit multiplicative constant} lifespan
for Schr\"odinger data with fixed $L^2$ norm larger than the ground
state mass.
\begin{theorem}[follows from \cite{glangetas_concentration_1994,glangetas_existence_1994}]
\label{thm:blowup}
For each $r>\|Q\|_{L^2}$ there exists $c>0$ such that for every $R\geq
r$ there exists a smooth solution $(u,n)$ with initial datum
$(u(0),n(0),\partial_t n(0))\in \mathbf{H}^{0,-\frac12}_{R,r}$ which
blows up at time $T:=cR^{-2}$, i.e.
\begin{equation}\label{eq:blowup}
\|n(t)\|_{H^{-\frac{1}{2}}}+\|\partial_tn(t)\|_{H^{-\frac{3}{2}}} \to \infty \quad (t \to T).
\end{equation}
\end{theorem}
The absence of the $L^2$ norm of $u$ in \eref{eq:blowup} is due to the $L^2$ conservation law. We refer the reader to \cite{glangetas_concentration_1994,glangetas_existence_1994} for further properties of the blow-up solutions such as $L^2$ norm concentration for $u$.
Finally, we state a result which shows the optimality of the imposed regularity assumptions in Theorem \ref{main}.
\begin{theorem}\label{thm:sharp}
  Assume there exists $0<r\leq R$ and $T>0$ such that the flow map
  $u_0\mapsto u$ for smooth data extends continuously to a map
$$
\mathbf{H}^{k,\ell}_{r,R} \to \mathbf{X}_{T}^{k,\ell}
$$
for some $\ell-2k+\frac12>0$ or $\ell<-\frac12$. Then this map
fails to be $C^2$ at the origin with respect to these norms.
\end{theorem}
The rest of the paper is organized as follows: In Section \ref{sect:not}, we set up the notation and introduce function spaces which we will use in the sequel. In Section \ref{sect:reduced} we outline the standard procedure (cp. \cite{ginibre_cauchy_1997}) for reducing \eref{eq:zs} to a first order (in time) system.
Section \ref{sect:mult-est} is devoted to the crucial multilinear estimates which are the main ingredients in the proof of Theorem \ref{main}. Section \ref{sect:proof} contains estimates for the linear group and the conclusion of the proof of Theorem \ref{main}. The counterexamples which lead to the sharpness result of Theorem \ref{thm:sharp} are constructed in Section \ref{sect:counter}, along with a proof of Theorem \ref{thm:blowup} (which is based on the results from \cite{glangetas_concentration_1994,glangetas_existence_1994}). In the Appendix we give an alternative proof of Proposition \ref{prop:trans_low_mod} which keeps the paper self-contained.
\ack
The authors would like to thank James Colliander, Hartmut Pecher and the referees for their helpful comments.

The first author has been supported by NSF grant DMS0738442.  The
second author has been supported by NSF grant DMS0354539.  The third
author has been partially supported by an NSF postdoctoral fellowship.
The fourth author acknowledges support from NSF grant DMS0801261.

\section{Notation and function spaces}\label{sect:not}
\noindent

We write $A\ls B$ if there is a harmless constant $c>0$ such that
$A\leq c B$. Moreover, we write $A\gs B$ iff $B\ls A$.  and $A\sim B$
iff $A\ls B$ and $A\gs B$. Throughout this work we will denote dyadic
numbers $2^n$ for $n \in \N$ by capital letters, e.g. $N=2^n,L=2^l,\ldots$.

Let $\psi\in C^\infty_0((-2,2))$ be an even, non-negative function
with the property $\psi(r)=1$ for $|r|\leq 1$. We use it to define a
partition of unity in $\R$,
\[
1 = \sum_{N \geq 1}\psi_N, \; \psi_1 = \psi, \;
\psi_N(r)=\psi\left(\frac{r}{N}\right)-\psi\left(\frac{2r}{N}\right),
\;  N=2^n \geq 2.
\]
Thus $\supp\psi_1\subset [-2,2]$ and $\supp{\psi}_N\subset
[-2N,-N/2]\cup[N/2,2N]$ for $N\geq 2$. For
$f:\R^2 \to \C$ we define the dyadic frequency localization operators
$P_N$ by
\begin{equation*}
  \mathcal{F}_x(P_N f)(\xi) =  \psi_N(|\xi|) \mathcal{F}_x f(\xi).
\end{equation*}
For $u:\R^2\times \R \to \C$ we define
$(P_N u)(x,t)=(P_N u(\cdot,t))(x)$. We will often write $u_N=P_Nu$
for brevity.  We denote the space-time Fourier support of $P_N$ by the
corresponding Gothic letter
\begin{eqnarray*}
  \mathfrak{P}_1=&\left\{(\xi,\tau) \in \R^2\times \R \mid |\xi|\leq 2 \right\},\\
  \mathfrak{P}_N=&\left\{(\xi,\tau) \in \R^2\times \R \mid N/2\leq
    |\xi|\leq 2N \right\}.
\end{eqnarray*}

Moreover, for dyadic $L\geq 1$ we define the modulation localization
operators
\begin{eqnarray}
  \mathcal{F} (S_L u)(\tau,\xi) &=  \psi_L(\tau+|\xi|^2)
\mathcal{F} u(\tau,\xi) \qquad \mbox{(Schr\"odinger case)}, \label{eq:s_mod}\\
  \mathcal{F} (W^\pm_L u)(\tau,\xi) &= \psi_L(\tau\pm |\xi|) \mathcal{F} u(\tau,\xi)  \qquad \mbox{(Wave case)},\label{eq:w_mod}
\end{eqnarray}
and the corresponding space-time Fourier supports
\begin{eqnarray*}
  \mathfrak{S}_1=&\left\{(\xi,\tau) \in \R^2\times \R \mid |\tau+|\xi|^2|\leq 2 \right\},\\
  \mathfrak{S}_L=&\left\{(\xi,\tau) \in \R^2\times \R \mid L/2\leq
    |\tau+|\xi|^2|\leq 2L \right\},
\end{eqnarray*}
respectively
\begin{eqnarray*}
  \mathfrak{W}^{\pm}_1=&\left\{(\xi,\tau) \in \R^2\times \R \mid |\tau \pm |\xi||\leq 2 \right\},\\
  \mathfrak{W}^{\pm}_L=&\left\{(\xi,\tau) \in \R^2\times \R \mid
    L/2\leq |\tau \pm |\xi||\leq 2L \right\}.
\end{eqnarray*}

We also define an equidistant partition of unity in $\R$,
\[
1 = \sum_{j \in \Z} \beta_j, \qquad \beta_j(s) = \psi(s-j) \left(
  \sum_{k \in  \Z} \psi(s-k) \right)^{-1}. 
\]
 Finally, for $A \in \N$ we define an equidistant partition of unity
on the unit circle,
\[
1 = \sum_{j = 0}^{A-1} \beta^A_j, \qquad 
\beta^A_j(\theta)=\beta_j\left(\frac{A\theta}{\pi}\right)+
\beta_{j-A}\left(\frac{A\theta}{\pi}\right)
\]
We observe that $\supp(\beta^A_j)\subset \Theta^A_j$, where
\begin{equation*}
  \Theta^A_j:=\left[\frac{\pi}{A}(j-2),\frac{\pi}{A}(j+2)\right]\cup \left[-\pi +\frac{\pi}{A}(j-2),-\pi +\frac{\pi}{A}(j+2)\right].
\end{equation*}
Next we introduce the angular frequency localization
operators $Q^A_j$,
\begin{equation*}
  \mathcal{F}_x(Q^A_j f)(\xi) =  \beta^A_j(\theta) \mathcal{F}_x f(\xi), \;\mbox{ where } \;\xi=|\xi|(\cos\theta,\sin\theta).
\end{equation*}
For $u:\R^2\times \R \to \C$, $(x,t)\mapsto u(x,t)$ we set $(Q^A_j
u)(x,t)=(Q^A_j u(\cdot,t))(x)$. These operators localize functions in
frequency to the sets
\begin{equation*}
  \mathfrak{Q}^A_j=\left\{(|\xi|\cos(\theta),|\xi|\sin(\theta),\tau) \in \R^2\times \R \mid  \theta \in \Theta^A_j\right\}.
\end{equation*}
For $A\in \N$ we can now decompose $u:\R^2\times\R\to \C$ as
\begin{equation*}
  u=\sum_{j=0}^{A-1} Q^{A}_j u.
\end{equation*}

Next we turn our attention to defining the spaces which play a crucial
role in our analysis. As explained in the introduction, for $k,\ell
\in \R$ and $T>0$ we define the space $\mathbf{X}^{k,\ell}_T$ as the
Banach space of all pairs of space-time distributions $(u,n)$
\begin{equation}
  \label{E:Xdef}
  \eqalign{
    u \in C( [0,T]; H^k(\mathbb{R}^2;\C)), \cr
    n \in C( [0,T]; H^\ell(\mathbb{R}^2;\R)) \cap C^1( [0,T];
    H^{\ell-1}(\mathbb{R}^2;\R)),}
\end{equation}
endowed with the standard norm defined via
\begin{equation}\label{eq:x-norm}
\fl \|(u,n)\|^2_{\mathbf{X}^{k,\ell}_T} = \|u\|_{L^\infty([0,T];H_x^k)}^2 + \|n\|_{L^\infty([0,T];H_x^{\ell})}^2 +  \|\partial_tn\|_{L^\infty([0,T];H_x^{\ell-1})}^2.
\end{equation}

Let $\sigma,b\in \R$, $1\leq p< \infty$. In connection to the operator
$i\partial_t + \Delta$ we define the Bourgain space $X^S_{\sigma,b,p}$
of all $u \in \mathcal{S}'(\R^2\times \R)$ for which the norm
$$
\|u\|_{X^S_{\sigma,b,p}}=\left(\sum_{N\geq
    1}N^{2\sigma}\left(\sum_{L\geq 1}L^{pb}
    \|S_LP_Nu\|_{L^2}^p\right)^{\frac2p}\right)^{\frac12}
$$
is finite. Similarly, to the half-wave operators $i\partial_t \pm
\langle \nabla\rangle$ we associate the Bourgain spaces $X^{W^\pm}_{\sigma,b,p}$ of
all $v \in \mathcal{S}'(\R^2\times \R)$ for which the norm
$$
\|v\|_{X^{W^\pm}_{\sigma,b,p}} = \left(\sum_{N\geq
    1}N^{2\sigma}\left(\sum_{L \geq 1}L^{pb}
    \|W^\pm_LP_Nu\|_{L^2}^p\right)^{\frac2p}\right)^{\frac12}
$$
is finite. For $p=\infty$ we modify the definition as usual:
\begin{eqnarray*}
  \|v\|_{X^{W^\pm}_{\sigma,b,\infty}} =& \left(\sum_{N\geq
      1}N^{2\sigma}\sup_{L\geq 1}L^{2b}
    \|W^\pm_LP_Nu\|_{L^2}^2\right)^{\frac12},\\
  \|u\|_{X^S_{\sigma,b,\infty}}=&\left(\sum_{N\geq
      1}N^{2\sigma}\sup_{L\geq 1}L^{2b}
    \|S_LP_Nu\|_{L^2}^2\right)^{\frac12}.
\end{eqnarray*}
In cases where the Schwartz space
$\mathcal{S}(\R^2\times \R)$ is not dense in $X^{W^\pm}_{\sigma,b,p}$ or
$X^{S}_{\sigma,b,p}$, respectively, we redefine the spaces and take the closure of $\mathcal{S}(\R^2\times \R)$ instead. Therefore, it is enough to prove all estimates in these spaces for Schwartz functions.

Notice that a change of $\tau \pm |\xi|$ to $\tau \pm \langle \xi \rangle$ in \eref{eq:w_mod} would lead to equivalent norms.

Finally, we define $X^W_{\sigma,b,p}$ similar to $X^{W\pm}_{\sigma,b,p}$ by means of replacing $\tau\pm |\xi|$ in \eref{eq:w_mod} by $|\tau|-|\xi|$.
$X^W_{\sigma,b,p}$ will only be used to describe the regularity for solutions of the full wave equation in Theorem \ref{main}.

For a normed space $B\subset \mathcal{S}'(\R^n\times \R;\C)$ of
space-time distributions we denote by $\overline{B}$ the space of
complex conjugates with the induced norm.

A calculation shows that $\overline{X^{W\pm}_{s,b,p}}=X^{W\mp}_{s,b,p}$.  By duality,
\begin{eqnarray}
  (\overline{X^S_{s,b,p}})^*&=X^{S}_{-s,-b,p'} \; ,\label{eq:dual1}\\
  (\overline{X^{W\pm}_{s,b,p}})^*&=X^{W\pm}_{-s,-b,p'}
  \;,\label{eq:dual2}
\end{eqnarray}
for $1\leq p<\infty$, $s,b \in \R$.

For $T>0$ we define the
space $B(T)$ of restrictions of distributions in $B$ to the set
$\R^n\times (0,T)$ with the induced norm
\[
\|u\|_{B(T)}=\inf \{ \| \tilde{u} \|_{B} : \; \tilde{u} \in B \mbox{ is an extension of } u \mbox{ to }\R^n \times \R\}.
\]

\section{The reduced system}\label{sect:reduced}
\noindent
For the Zakharov system there is a standard procedure to factor the wave operator in order to
derive a first order system. In this section we outline the approach described in \cite{ginibre_cauchy_1997}.

Suppose that $(u,n)$ is a sufficiently regular solution to \eref{eq:zs}.
We define $\langle \nabla \rangle = (1-\Delta)^\frac12$ and
$v=n + i \langle \nabla \rangle^{-1} \partial_t n$ and obtain the system
\begin{equation}\label{eq:zs-red}
    \eqalign{
      & i\partial_t u + \Delta u = (\Re v) u, \\
      & i\partial_t v - \langle \nabla \rangle v= -\frac{\Delta}{\langle \nabla\rangle} |u|^2 - \langle \nabla \rangle^{-1} \Re v.}
\end{equation}
Given a solution $(u,v)$ to \eref{eq:zs-red} with initial data $(u_0,v_0)$, we obtain a solution to the original
system \eref{eq:zs} by setting $n=\Re v$.

In the following sections we will study the system \eref{eq:zs-red} and prove a well-posedness result for this system since it is slightly more convenient to iterate the reduced system \eref{eq:zs-red} instead of \eref{eq:zs} for symmetry reasons.

We call a pair of distributions $(u,n)$ a solution to \eref{eq:zs} if
\begin{equation}\label{eq:nv}
(u,n + i \langle \nabla \rangle^{-1} \partial_t n)
\end{equation}
is a solution of \eref{eq:zs-red} in the sense of the integral
equation \eref{eq:sol}. The uniqueness class $\mathbf{X}_T$ in the
statement of Theorem \ref{main} can be chosen as all $(u,n)$ such that
$u \in X^S_{0,\frac12,1}(T)$, $n \in X^W_{-\frac12,\frac12,1}(T)$
and $\partial_t n\in X^W_{-\frac32,\frac12,1}(T)$, see  Section
\ref{sect:not} for definitions.

We now reformulate the statement of Theorem \ref{main} into a similar 
statement about the reduced system \eref{eq:zs-red}. 

From the above relation \eref{eq:nv} between $v$ and $n$ and the
definitions it follows that if $v \in X^{W+}_{-\frac12,\frac12,1}(T)$
is a solution to \eref{eq:zs-red} then we have $n = \Re{v} \in
X^{W}_{-\frac12,\frac12,1}(T)$; since $\partial_t n =\langle \nabla
\rangle \Im v$ it also follows that $\partial_t n \in
X^W_{-\frac32,\frac12,1}(T)$.  Conversely, if $n \in
X^{W}_{-\frac12,\frac12,1}(T)$ and $\partial_t n \in
X^W_{-\frac32,\frac12,1}(T)$ then a straightforward computation shows that
$v \in X^{W+}_{-\frac12,\frac12,1}(T)$.

The above considerations allow us to claim the statement of Theorem \ref{main} by 
a proving a similar statement about the reduced system \eref{eq:zs-red} with initial data $(u_0,v_0) \in L^2 \times H^{-\frac12}$.
Obviously, in the context of \eref{eq:zs-red} we adjust
the definition of $\mathbf{X}_T$ to $X^S_{0,\frac12,1}(T) \times X^{W+}_{-\frac12,\frac12,1}(T)$. 

We finish the section with a simple remark. According to the linear part of the equation of $v$ in \eref{eq:zs-red},
the corresponding $X^{W+}_{s,b,p}$ spaces should have been defined with the weight $\tau + \langle \xi \rangle$ instead
of $\tau + |\xi|$. However, a direct computation shows that the two spaces are the same. The reason behind it is
that we deal with local theory $T \leq 1$ and inhomogeneous norms. 

In the sequel of the paper we will restrict our attention to the reduced system \eref{eq:zs-red}.

\section{Multilinear estimates}\label{sect:mult-est}
\noindent
This section is devoted to the proof of the following Theorem.
\begin{theorem}\label{thm:tri}
  For all $0<T\leq 1$ and for all functions $u,u_1,u_2 \in
  X^S_{0,\frac{5}{12},1}(T)$ and $v\in
  X^{W+}_{-\frac12,\frac{5}{12},1}(T)$ the following estimates hold
  true:
  \begin{eqnarray}
    \| uv\|_{X^S_{0,-\frac{5}{12},\infty}(T)} \ls{} \|u\|_{X^S_{0,\frac{5}{12},1}(T)}\|v\|_{X^{W+}_{-\frac12,\frac{5}{12},1}(T)}\label{eq:sws1},\\
    \| u\bar{v}\|_{X^S_{0,-\frac{5}{12},\infty}(T)} \ls{} \|u\|_{X^S_{0,\frac{5}{12},1}(T)} \|v\|_{X^{W+}_{-\frac12,\frac{5}{12},1}(T)} \label{eq:sws2},\\
    \left\| \frac{\Delta}{\langle\nabla\rangle} ( u_1
    \bar{u}_2)\right\|_{X^{W+}_{-\frac12,-\frac{5}{12},\infty}(T)} \ls{} 
    \|u_1\|_{X^S_{0,\frac{5}{12},1}(T)}
    \|u_2\|_{X^S_{0,\frac{5}{12},1}(T)}.\label{eq:wss}
  \end{eqnarray}
\end{theorem}

We introduce the notation
\[
I(f,g_1,g_2) =  \int f(\zeta_1-\zeta_2)
       g_1(\zeta_1) g_2(\zeta_2)
        d\zeta_1 d\zeta_2,
\]
where $\zeta_i=(\xi_i,\tau_i)$, $i=1,2$.
Using \eref{eq:dual1} and \eref{eq:dual2} and the fact that
$\overline{\mathcal{F}u}=\mathcal{F}\overline{u}(-\cdot)$, we can
reduce Theorem \ref{thm:tri} to the following trilinear estimate:
\begin{proposition}\label{prop:trilinear}
  For all $v,u_1,u_2 \in \mathcal{S}(\R^2\times \R)$ it holds
  \begin{equation}\label{eq:trilinear}
     \left| I(\mathcal{F}v,
        \mathcal{F}u_1, \mathcal{F}u_2)
        \right|
       \ls{} \|u_1\|_{X^{S}_{0,\frac{5}{12},1}}
      \|u_2\|_{X^{S}_{0,\frac{5}{12},1}}\|v\|_{X^{W\pm}_{-\frac12,\frac5{12},1}}.
  \end{equation}
\end{proposition}

The proof of Proposition \ref{prop:trilinear} is given at the end of
this section.  As building blocks we provide a number of preliminary
estimates first. These are concerned with functions which are
dyadically localized in frequency and modulation. In some cases we
additionally differentiate frequencies by their angular separation.

We start this analysis by recalling the well-known bilinear
generalization of the linear $L^4$ Strichartz estimate for the
Schr\"odinger equation due to Bourgain \cite[Lemma
111]{bourgain_refinements_1998}, see \eref{eq:str-schr-schr}
below. We observe that a similar estimate is true for a
Wave-Schr\"odinger interaction.
\begin{proposition}[Bilinear Strichartz
  estimates]
\
\begin{enumerate}
\item Let $v_1,v_2\in L^2(\R^3)$ be dyadically Fourier-localized such that
\[
\supp \mathcal{F} v_i\subset \mathfrak{P}_{N_i}\cap \mathfrak{S}_{L_i}
\]
for $L_1,L_2\geq 1$, $N_1,N_2\geq 1$. Then the following estimate
holds:
\begin{equation}
  \|v_1v_2\|_{L^2(\R^3)}\ls \left(\frac{N_1}{N_2}\right)^{\frac12}L_1^\frac12L_2^\frac12\|v_1\|_{L^2}\|v_2\|_{L^2}.\label{eq:str-schr-schr}
\end{equation}
\item Let $u,v\in L^2(\R^3)$ be such that
\[
\supp \mathcal{F} u \subset C\times \R \cap \mathfrak{W}^\pm_{L},\;
\supp \mathcal{F} v \subset \mathfrak{P}_{N_1}\cap \mathfrak{S}_{L_1}
\]
for $L,L_1\geq 1$, $N_1\geq 1$ and a cube $C\subset \R^2$ of
side length $d\geq 1$.
Then the following estimate holds:
\begin{equation}
  \|uv\|_{L^2(\R^3)}\ls
  \left(\frac{\min\{d,N_1\}}{N_1}\right)^{\frac12}L^\frac12
  L_1^\frac12 \|u\|_{L^2}\|v\|_{L^2}.
\label{eq:str-wave-schr-gen}
\end{equation}
 In particular, if
\[
\supp \mathcal{F} u \subset \mathfrak{P}_{N} \cap
\mathfrak{W}^\pm_{L}, \;
\supp \mathcal{F} v \subset \mathfrak{P}_{N_1}\cap \mathfrak{S}_{L_1}
\]
for $L,L_1\geq 1$, $N,N_1\geq 1$, it follows
\begin{equation}
  \|uv\|_{L^2(\R^3)}\ls \left(\frac{\min\{N,N_1\}}{N_1}\right)^{\frac12}L^\frac12 L_1^\frac12. \|u\|_{L^2}\|v_1\|_{L^2}.\label{eq:str-wave-schr}
\end{equation}
\end{enumerate}
On the left hand side of \eref{eq:str-schr-schr},
\eref{eq:str-wave-schr} and \eref{eq:str-schr-schr} we may replace
each function with its complex conjugate.
\label{prop:bilinear-str} \end{proposition}
\begin{proof}
  As remarked above the estimate \eref{eq:str-schr-schr} is provided
  by \cite[Lemma 111]{bourgain_refinements_1998}, so it remains to
  show \eref{eq:str-wave-schr-gen} and \eref{eq:str-wave-schr}.
  With $f=\mathcal{F}u$ and $g=\mathcal{F}v$ it follows
  \begin{equation*}\left\|\int
      f(\xi_1,\tau_1)g(\xi-\xi_1,\tau-\tau_1)d\xi_1d\tau_1
    \right\|_{L^2_{\xi,\tau}}
    \ls \sup_{\xi,\tau}|E(\xi,\tau)|^\frac12 \|f\|_{L^2}\|g\|_{L^2}
  \end{equation*}
  by the Cauchy-Schwarz inequality, where
  \begin{equation*}
    E(\xi,\tau)=\{(\xi_1,\tau_1)\in \supp f \mid (\xi-\xi_1,\tau-\tau_1)\in \supp g \}\subset \R^3.
  \end{equation*}
  With $\underline{l}=\min\{L,L_1\}$ and $\overline{l}=\max\{L,L_1\}$
  the volume of this set can be estimated as
  \begin{equation*}
    |E(\xi,\tau)|
    \leq  \underline{l} \cdot
    | \{\xi_1\mid | \tau\pm |\xi_1|+|\xi-\xi_1|^2|\ls \overline{l}, \xi_1\in C,|\xi-\xi_1|\sim N_1\}|,
  \end{equation*}
  by Fubini's theorem.  The latter subset of $\R^2$ is contained in a
  cube of side length $m$, where $m \sim \min\{d,N_1\}$, so if $N_1=1$
  the estimate follows. If $N_1\geq 2$ and the first component
  $\xi_{1,1}$ is fixed, then the second component $\xi_{1,2}$ is
  confined to an interval of length $m$, and vice versa. In the subset
  where $|(\xi-\xi_{1})_2|\gs N_1$ we observe that
  $|\partial_{\xi_{1,2}}(\tau\pm|\xi_1|+|\xi-\xi_1|^2)|\gs N_1$, and
  similarly in the subset where $|(\xi-\xi_{1})_1|\gs N_1$ we
  observe that
  $|\partial_{\xi_{1,1}}(\tau\pm|\xi_1|+|\xi-\xi_1|^2)|\gs N_1$. This
  shows that
\[
| \{\xi_1\mid | \tau\pm |\xi_1|+|\xi-\xi_1|^2|\ls \overline{l}, \xi_1
\in C,|\xi-\xi_1|\sim N_1\}|\ls N_1^{-1}\overline{l}m,
\]
and the claim \eref{eq:str-wave-schr-gen} follows. This also implies
the claim \eref{eq:str-wave-schr} because the dyadic annulus of
radius $N$ is contained in a cube of side length $d\sim N$.
\end{proof}

\noindent
Let $\angle(\xi_1,\xi_2)\in [0,\frac{\pi}{2}]$ denote the (smaller) angle between
the lines spanned by $\xi_1,\xi_2 \in \R^2$.  For dyadic numbers $64 \leq A \leq M$
we consider the following angular decomposition
\begin{eqnarray}
\fl  \R^2\times \R^2 \nonumber
  &=\left\{ \angle(\xi_1,\xi_2)\leq \frac{16\pi}{M}\right\}\cup
\bigcup_{64 \leq A\leq M}\left\{ \frac{16\pi}{A} \leq \angle(\xi_1,\xi_2)\leq \frac{32\pi}{A}\right\}
\nonumber\\
\fl &=\bigcup_{0\leq j_1,j_2\leq M-1 \atop |j_1-j_2|\leq 16}
  \mathfrak{Q}^M_{j_1}\times \mathfrak{Q}^M_{j_2} \cup
   \bigcup_{64\leq A\leq M}
   \bigcup_{0\leq j_1,j_2\leq A-1 \atop 16\leq |j_1-j_2|\leq 32}
  \mathfrak{Q}^A_{j_1}\times \mathfrak{Q}^A_{j_2}\label{eq:dyadic_angular}
\end{eqnarray}
Therefore, we consider for each dyadic $A\in [64,M]$ slices of angular
aperture $\sim A^{-1}$ with an angular separation of size $\sim
A^{-1}$, and additionally slices which are of angular aperture less
than $M^{-1}$. This is a dyadic, angular Whitney type decomposition
with threshold $M$.

\begin{proposition}[Transverse high-high interactions, low
  modulation]\label{prop:trans_low_mod}
 Let $f,g_1,g_2\in L^2$ with
  $\| f \|_{L^2} =\| g_1\|_{L^2}= \|g_2\|_{L^2}=1$ and
  \begin{equation*}
    \supp(f)\subset \mathfrak{W}^\pm_L \cap \mathfrak{P}_{N},\quad 
    \supp(g_k)\subset \mathfrak{Q}^{A}_{j_k} \cap \mathfrak{P}_{N_k} \cap \mathfrak{S}_{L_k}\quad (k=1,2).
  \end{equation*}
where the frequencies $N,N_1,N_2$ and modulations $L,L_1,L_2$ satisfy
\[
64\leq N\ls N_1 \sim N_2, \qquad L_1,L_2,L \ls N_1^2
\]
while the angular localization parameters $A$ and $j_1,j_2$ satisfy
\[
64\leq A\ll N_1, \qquad 16 \leq |j_1-j_2| \leq 32
\]
 Then the following estimate holds
  \begin{equation} \label{eq:trans_low_mod}
    |I(f,g_1,g_2)|
    \ls \frac{1}{N_1^\frac12} \left(\frac{A}{N_1}\right)^{\frac12}
    (L_1 L_2 L)^\frac12.
  \end{equation}
\end{proposition}
The following proof of Proposition \ref{prop:trans_low_mod} is based
on a quantitative, nonlinear version of
the classical Loomis-Whitney-inequality
\cite{loomis_inequality_1949}.

\begin{proposition}[see \cite{bejenaru_convolution_2008}]\label{cor:conv_theta}
  Let $C_1,C_2,C_3$ be cubes in $\R^3$ of diameter $2R>0$. Consider two
  paraboloids in $\R^3$ which are graphs of $\phi_1,\phi_2\in C^{1,1}$ within $C_1,C_2$ and a cone in $\R^3$ which is a graph of $\phi_3 \in C^{1,1}$ within $C_3$,
  such that the homogeneous semi-norms satisfy $[\phi_j]_{C^{1,1}}\ls 1$. Moreover, assume that
  they are transversal in the sense that
  the determinant of every triple of unit normals to points on the surfaces within these cubes is at least
  of size $\theta>0$ and suppose that $R\ls \theta$.
  Now, for given subsets $\Sigma_1,\Sigma_2,\Sigma_3$ of the above surfaces which are contained in
  the $\frac12$-shrinked cubes with same center
  and for each $f\in L^2(\Sigma_1)$ and $g\in L^2(\Sigma_2)$ the
  restriction of the convolution $f\ast g$ to $\Sigma_3$ is a
  well-defined $L^2(\Sigma_3)$-function which satisfies
  \begin{equation}\label{eq:conv_theta}
    \|f\ast g\|_{L^2(\Sigma_3)}\leq  \frac{C}{\sqrt{\theta}}
    \|f\|_{L^2(\Sigma_1)}\|g\|_{L^2(\Sigma_2)}.
  \end{equation}
\end{proposition}
This follows from \cite[Corollary 1.6]{bejenaru_convolution_2008}.
We also refer the interested reader to
the earlier paper \cite{bennett_nonlinear_2005} which contains a version of the aforementioned inequality in broader
generality under slightly more restrictive and
non-scalable assumptions. To keep the paper
self-contained, we provide an independent proof of Proposition \ref{prop:trans_low_mod} in \ref{sect:alt_proof} which is based on elementary geometric considerations and orthogonality.

\begin{proof} We abuse notation and replace $g_2$ by $g_2(-\cdot)$ and
  change variables $\zeta_2\mapsto -\zeta_2$ to obtain the usual
  convolution structure. From now on it holds $|\tau_2-|\xi_2|^2|\sim
  L_2$ within the support of $g_2$. We consider only the case $\supp(f)\subset \mathfrak{W}^-_L$
  since in the case $\supp(f)\subset \mathfrak{W}^+_L$ the same arguments apply.

  For fixed $\xi_1,\xi_2$ we change variables $c_1=\tau_1+|\xi_1|^2$,
  $c_2=\tau_2-|\xi_2|^2$.  By decomposing $f$ into $\sim L$ pieces and
  applying the Cauchy-Schwarz inequality, it suffices to prove
  \begin{equation}\label{eq:reduce}
    \eqalign{
      &\left| \int g_{1}(\phi^-_{c_1}(\xi_1)) g_{2}
        (\phi^+_{c_2}(\xi_2))
        f(\phi^-_{c_1}(\xi_1)+\phi^+_{c_2}(\xi_2)) d\xi_1 d\xi_2 \right| \\
      \ls{} & \frac{A^{\frac12}}{N_1} \|
      g_{1}\circ\phi^-_{c_1}\|_{L^2_\xi} \| g_{2}\circ \phi^+_{c_2}
      \|_{L^2_\xi} \| f\|_{L^2}}
  \end{equation}
  where $f$ is now supported in $c \leq \tau-|\xi| \leq c+1$ and
  $\phi^\pm_{c_k}(\xi)=(\xi,\pm|\xi|^2 + c_k)$, $k=1,2$, and the
  implicit constant is independent of $c, c_1,c_2$.

  We refine the localization on $\xi$ and $\tau$ components by
  orthogonality methods, see also Lemma \ref{L:localize}.  Since the
  support of $f$ in the $\tau$ direction is confined to an interval of
  length $\ls N_1$, $|\xi_2|^2-|\xi_1|^2$ is localized
  in a specific interval of length $\sim N_1$ which in turn localizes
  $|\xi_2|-|\xi_1|$ in an interval of size $\sim 1$. By decomposing
  the plane into annuli of size $\sim 1$ and using the Cauchy-Schwarz
  inequality, we reduce \eref{eq:reduce} further to the additional
  assumption that the support of $g_1 \circ \phi^-_{c_1}$ and 
  $g_2 \circ \phi^+_{c_2} $ is an interval of
  length $\sim 1 \ls N_1 A^{-1}$.  Recalling the additional angular
  localization, we can assume that $g_1,g_2$ and $f$ are each
  localized in cubes of size $N_1 A^{-1}$ with respect to the $\xi$
  variables.
        
  We use the parabolic scaling $(\xi,\tau)\mapsto (N_1 \xi, N_1^2
  \tau)$ to define
  \begin{equation*}
    \tilde{f}(\xi,\tau)=f(N_1 \xi,N_1^2\tau), \;
    \tilde{g}_k(\xi_k,\tau_k)=g_{k}(N_1 \xi_k,N_1^2\tau_k),\; k=1,2.
  \end{equation*}
  If we set $\tilde{c}_k=c_kN_k^{-2}$, equation \eref{eq:reduce}
  reduces to
  \begin{equation} \label{eq:reduce1}
    \eqalign{
      &\left| \int \tilde{g}_{1}(\phi^-_{\tilde{c}_1}(\xi_1))
        \tilde{g}_{2} (\phi^+_{\tilde{c}_2}(\xi_2))
        \tilde{f}(\phi^-_{\tilde{c}_1}(\xi_1)+\phi^+_{\tilde{c}_2}(\xi_2)) d\xi_1 d\xi_2 \right| \\
      \ls{}& \frac{A^{\frac12}}{N_1} \|
      \tilde{g}_{1}\circ\phi^-_{\tilde{c}_1} \|_{L^2_\xi} \|
      \tilde{g}_{2} \circ\phi^+_{\tilde{c}_2} \|_{L^2_\xi} \|
      \tilde{f} \|_{L^2},}
  \end{equation}                
  where now $\tilde{g}_k$ is supported in a cube of size $\sim A^{-1}$
  with $|\xi_k| \sim 1$ and the supports are separated by $\sim
  A^{-1}$.  $\tilde{f}$ is supported in a neighborhood of size
  $N^{-2}_1$ of the surface $S_3$ parametrized by
  $(\xi,\psi_{N_1}(\xi))$ for $\psi_{N_1}(\xi)=\frac{|\xi|}{N_1} +
  \frac{c}{N^2_1}$.  Let us put $\varepsilon=N_1^{-2}$ and denote this
  neighborhood by $S_3(\varepsilon)$.  The separation of $\xi_1$ and
  $\xi_2$ above implies also that in the support of $\tilde{f}$ we
  have $|\xi| \gs A^{-1}\geq N_1^{-1}$.

  By density and duality it is enough to consider continuous $\tilde
  g_1,\tilde g_2$ and we can further rewrite the above estimate as
  \begin{equation} \label{surface} \| \tilde g_1|_{S_1} \ast \tilde
    g_2|_{S_2} \|_{L^2(S_3(\varepsilon))} \ls A^\frac12
    \varepsilon^\frac12 \| \tilde g_1 \|_{L^2(S_1)} \| \tilde g_2
    \|_{L^2(S_2)}
  \end{equation}
  where $S_i$, $i=1,2$ are parametrized by $\phi^\pm_{\tilde c_i}$.
  The above localization properties of the support of $\tilde{g}_i$
  are inherited by $S_i$, which implies that the maximal diameter of
  the $S_1$, $S_2$ and $S_3$ is at most $R\sim A^{-1}$.  Obviously,
  the parametrizations of the paraboloids $S_1$ and $S_2$ have $C^{1,1}$
  semi-norm $\sim 1$.  Concerning
  $S_3$ we estimate
  $$
  |\nabla \psi_{N_1} (\xi)- \nabla \psi_{N_1} (\eta)| \ls N_1^{-1} |
  \frac{\xi}{|\xi|}-\frac{\eta}{|\eta|}| \ls |\xi - \eta|
  $$
  where we have used that $|\xi|,|\eta| \geq N_1^{-1}$ in
  the base of $S_3$. Therefore, the $C^{1,1}$ semi-norm for our parametrization of $S_3$ is $ \ls 1$.
  
  Finally, we need to analyze the transversality properties of our surfaces. In other words, we
  need to determine a uniform lower bound $\theta$ on the size of the
  determinant $d$ of the matrix of three unit normal vector fields.
  Intuitively it is clear that -- since the parabolically rescaled
  cone is almost flat -- this is determined by the minimal
  angular separation $\sim A^{-1}$ between the $\xi$-supports of $g_1$
  and $g_2$. In fact, we will show that $\theta\gs A^{-1}$ below.  In
  summary, we have $R \ls \theta$ and we invoke \eref{eq:conv_theta} to obtain \eref{surface}.

  Let us carefully verify the transversality condition $\theta\gs A^{-1}$ indicated above: The
  determinant of any three unit normals to $S_1$, $S_2$, and $S_3$ is given by
  \begin{equation*}
    d=\left|\begin{array}{ccc}
      \frac{2\xi_1}{\la 2\xi \ra} & \frac{2\eta_1}{\la 2\eta \ra} & \frac{\zeta_1}{|\zeta| \la N_1 \ra} \\
      \frac{2\xi_2}{\la 2\xi \ra} & \frac{2\eta_2}{\la 2\eta \ra} & \frac{\zeta_2}{|\zeta| \la N_1 \ra} \\
      \frac{1}{\la 2\xi \ra} & -\frac{1}{\la 2\eta \ra} & \frac{N_1}{\la N_1 \ra}
    \end{array}\right|
  \end{equation*}
  which we expand as $d=d_1+d_2+d_3$, with main contribution
  \begin{equation*}
    d_1= \frac{N_1}{\la N_1 \ra} \left|\begin{array}{cc} \frac{2\xi_1}{\la 2\xi \ra} & \frac{2\eta_1}{\la 2\eta \ra}  \\\frac{2\xi_2}{\la 2\xi \ra} & \frac{2\eta_2}{\la 2\eta \ra} 
      \end{array}\right|
    \end{equation*}
  and the error terms
  \begin{equation*}
    d_2=-\frac{\zeta_2}{|\zeta| \la N_1 \ra}  \left|\begin{array}{cc}
      \frac{2\xi_1}{\la 2\xi \ra} & \frac{2\eta_1}{\la 2\eta \ra}  \\
      \frac{1}{\la 2\xi \ra} & - \frac{1}{\la 2\eta \ra} \end{array}\right| ,\quad 
    d_3= \frac{\zeta_1}{|\zeta| \la N_1 \ra}  \left|\begin{array}{cc}
      \frac{2\xi_2}{\la 2\xi \ra} & \frac{2\eta_2}{\la 2\eta \ra} \\
      \frac{1}{\la 2\xi \ra} & -\frac{1}{\la 2\eta \ra} 
\end{array}\right|.
  \end{equation*}
  The contribution of the last two terms $d_2$ and $d_3$ is bounded by
  \[
  |d_2|+|d_3| \ls \frac{|\zeta_1|+|\zeta_2|}{|\zeta| \la N_1\ra} \ls
  N_1^{-1}
  \]
  The first determinant $d_1$ can be rewritten as
  $$
  d_1= \frac{N_1}{\la N_1 \ra} \frac{2|\xi|}{\la 2\xi \ra}
  \frac{2|\eta|}{\la 2\eta \ra} \left|\begin{array}{cc} \frac{\xi_1}{|\xi|} &
    \frac{\eta_1}{|\eta|} \\ \frac{\xi_2}{|\xi|} &
    \frac{\eta_2}{|\eta|}
  \end{array}\right|
  =\frac{N_1}{\la N_1 \ra} \frac{2|\xi|}{\la 2\xi \ra}
  \frac{2|\eta|}{\la 2\eta
    \ra}\sin\angle\left(\frac{\xi}{|\xi|},\frac{\eta}{|\eta|}\right)
  $$
  Recalling that $|\xi|,|\eta| \sim 1$ (since they are in the support
  of $g_1$, respectively $g_2$), it follows that $\frac{N_1}{\la N_1
    \ra} \frac{2|\xi|}{\la 2\xi \ra} \frac{2|\eta|}{\la 2\eta \ra} \gs
  1$.  By the angular separation between $S_1$ and $S_2$ we obtain $
  |d_1| \gs A^{-1}$ and by recalling that $A \gg N_1$ it follows that
  $|d| \gs A^{-1}$.
\end{proof}

In the case where the maximal modulation is high a different bound
will be favourable.

\begin{proposition}[Transverse high-high interactions, high
  modulation]\label{prop:trans_high_mod}
  Let $f,g_1,g_2\in L^2$, $\|f\|_{L^2}=\|g_1\|_{L^2}=\|g_2\|_{L^2}=1$
  such that
  \begin{equation*}
    \supp(f)\subset \mathfrak{P}_{N}\cap \mathfrak{W}^\pm_L ,\quad 
    \supp(g_k)\subset \mathfrak{Q}^{A}_{j_k} \cap \mathfrak{P}_{N_k} \cap \mathfrak{S}_{L_k}\quad (k=1,2),
  \end{equation*}
  with $64 \leq N \ls N_1 \sim N_2$ and $64\leq A \leq N_1$. Moreover,
  assume that $16\leq |j_1-j_2|\leq 32$.
 Then
  \begin{equation}\label{eq:trans_high_mod}
    |I(f,g_1,g_2)| 
    \ls \frac{L_1^\frac12 L_2^\frac12L^\frac12 N^{-\frac12}}{\max\{L,L_1,L_2\}^\frac12}\left(\frac{N_1}{A}\right)^{\frac12}
  \end{equation}
\end{proposition}
\begin{remark}
 The estimate \eref{eq:trans_high_mod} gives a better bound than
  \eref{eq:trans_low_mod} in the case where
  \begin{equation}\label{eq:high_mod_cond}
    \max\{L,L_1,L_2\}\geq \left(\frac{N_1}{A}\right)^2\frac{N_1}{N}.
  \end{equation}
\end{remark}
\begin{proof}[Proof of Proposition \ref{prop:trans_high_mod}]
  After a rotation we may assume that $j_1=0$.  Due to the
  localization of the wedges we observe that the integral vanishes
  unless $N\gs N_1A^{-1}$, since $|\xi_{2,2}-\xi_{1,2}|\sim N_1
  A^{-1}$. We consider two cases:
  \begin{enumerate}
  \item\label{it:smallN} $N \sim N_1A^{-1}$
  \item\label{it:largeN}$N\gg N_1 A^{-1}$.
  \end{enumerate}
  In case \eref{it:smallN}  we start with the subcase
  where $\max\{L,L_1,L_2\}=L$. From  the bilinear Strichartz
  estimate for the Schr\"odinger equation \eref{eq:str-schr-schr},
  using $N\sim A^{-1}N_1$, we obtain
  \begin{equation*}
     |I(f,g_1,g_2)|
    \ls 
    \left(L_1L_2\right)^{\frac12}\|f\|_{L^2}\|g_1\|_{L^2}\|g_2\|_{L^2}.
  \end{equation*}
  The subcases where $\max\{L,L_1,L_2\}=L_i$ for $i=1,2$ follow in the
  same way by using \eref{eq:str-wave-schr} instead of
  \eref{eq:str-schr-schr}.

  In Case \eref{it:largeN} we also start with the subcase where
  $\max\{L,L_1,L_2\}=L$. Without any restriction in generality
assume also that $L_1 \leq L_2$.
Denoting
\[
\chi = 1_{\mathfrak{Q}^{A}_{j_1} \cap \mathfrak{P}_{N_1} \cap \mathfrak{S}_{L_1}}
 1_{\mathfrak{Q}^{A}_{j_2} \cap \mathfrak{P}_{N_2} \cap \mathfrak{S}_{L_2}}
\]
we use Cauchy-Schwarz to estimate
\begin{eqnarray*}
\fl \left|\int f(\zeta_1-\zeta_2)g_1(\zeta_1) g_2(\zeta_2)
 d\zeta_1 d\zeta_2\right| 
&\lesssim{} & \| \chi f(\zeta_1 - \zeta_2)\|_{L^2}
\|g_1(\zeta_1) g_2(\zeta_2)\|_{L^2} \\
&\lesssim{} & \sup_{\zeta_0 \in 
\mathfrak{P}_{N}\cap \mathfrak{W}^\pm_L}
|B(\zeta_0)|^\frac12 \|f\|_{L^2}\|g_1\|_{L^2}\|g_2\|_{L^2}
\end{eqnarray*}
where
\[
B(\zeta_0) = \{ \zeta_1 \; | \; \zeta_1 \in  \mathfrak{Q}^{A}_{j_1} \cap 
\mathfrak{P}_{N_1} \cap
\mathfrak{S}_{L_1}; \zeta_1 - \zeta_0 \in \mathfrak{Q}^{A}_{j_2} \cap \mathfrak{P}_{N_2} \cap
\mathfrak{S}_{L_2}\}.
\]
To bound the size of the set $B(\zeta_0)$ we observe that 
for $\zeta_0=(\xi_0,\tau_0)$ and $\zeta_1=(\xi_1,\tau_1)$ 
as above we must have $|\xi_{0,1}| \sim N$ and
\[
|\tau_1 - \xi_1^2| \lesssim L_1, \qquad |\xi_{1,2}| \lesssim \frac{N_1}A, \qquad |\tau_1-\tau_0+|\xi_1-\xi_0|^2|\sim L_2.
\]
Since $\partial_{\xi_{1,1}}(|\xi_1|^2-|\xi_1-\xi_0|^2)=2\xi_{0,1}$
which has size $N$, it follows that
\begin{equation}
|B(\xi_0,\tau_0)|\ls L_1 \frac{L_2}N \frac{N_1}{A}
\label{bsize}\end{equation}
and the conclusion of the Proposition follows.

Let us now assume that $\max\{L,L_1,L_2\}=L_1$; the subcase
when $\max\{L,L_1,L_2\}=L_2$ is similar. Using Cauchy-Schwarz
as above we obtain
\[
  |I(f,g_1,g_2)|
    \ls 
 \sup_{\zeta_1 \in \mathfrak{Q}^{A}_{j_1} \cap 
\mathfrak{P}_{N_1} \cap \mathfrak{S}_{L_1}}
|C(\zeta_1)|^\frac12
    \|f\|_{L^2}\|g_1\|_{L^2}\|g_2\|_{L^2},
\]
where
\[
C(\zeta_1) = \{ \zeta_2 \; | \; \zeta_2 \in  \mathfrak{Q}^{A}_{j_2} \cap 
\mathfrak{P}_{N_2} \cap \mathfrak{S}_{L_2}; \zeta_1 - \zeta_2 \in 
\mathfrak{P}_{N}\cap \mathfrak{W}^\pm_L\}.
\]
Setting $\underline{l}=\min\{L,L_2\}$ and
$\overline{l}=\max\{L,L_2\}$, we observe that 
given $\xi_2$, $\tau_2$ can only range in an interval of size
$\lesssim \underline{l}$. On the other hand, for $\xi_2$ we have the
restrictions
\[
 |\xi_{2,2}|\ls \frac{N_1}{A}, \qquad
 |\tau_1+|\xi_2|^2\pm |\xi_1-\xi_2||\ls
 \overline{l}.
\]
Since $|\partial_{\xi_{2,1}}(|\xi_2|^2\pm
|\xi_1-\xi_2|)|=2|\xi_{2,1}|\gs N_1$, we obtain
\begin{equation}\label{eq:setc}
|C(\zeta_1)| \lesssim  \underline{l}
\frac{\overline{l}}{N_1}\frac{N_1}{A}= \frac{LL_2}{A}
 \end{equation}
again concluding the proof of the Proposition.
\end{proof}

Next, we consider the case where the frequencies $\xi_1$ and $\xi_2$
are almost parallel.  This can be viewed as an almost one-dimensional
interaction.

\begin{proposition}[Parallel high-high interactions]\label{prop:paral_hh}
  Let $f,g_1,g_2\in L^2$, $\|f\|_{L^2}=\|g_1\|_{L^2}=\|g_2\|_{L^2}=1$
  such that
  \begin{equation*}
    \supp(f)\subset \mathfrak{P}_{N}\cap \mathfrak{W}^\pm_L ,\quad 
    \supp(g_k)\subset \mathfrak{Q}^{A}_{j_k} \cap \mathfrak{P}_{N_k} \cap \mathfrak{S}_{L_k}\quad (k=1,2),
  \end{equation*}
  with $1 \ll N \ls N_1 \sim N_2$. Assume that $A \sim N_1$
  and $|j_1-j_2|\leq 16$.  Then for all
  $L,L_1,L_2\geq 1$ we have
  \begin{equation}\label{eq:paral_hh}
 |I(f,g_1,g_2)|   
    \ls L_{1}^{\frac{5}{12}}L_{2}^{\frac{5}{12}}L^{\frac{5}{12}}\frac{1}{N^\frac12}
    \left( \frac{N}{N_1}\right)^{\frac14}
  \end{equation}
\end{proposition}
\begin{proof}
  After a rotation we may assume that $j_1=0$.  Due to the
  localization of the wedges we observe that 
$|\xi_{0,2}|,|\xi_{1,2}|,|\xi_{2,2}| \lesssim 1$.
This shows that $|\xi_{1,1} -\xi_{2,1}|= |\xi_{0,1}| \sim N$,
$|\xi_{1,1}|,|\xi_{2,1}| \sim N_1$. In addition, we must have
\[
||\xi_1 - \xi_2| \pm (|\xi_1|^2 - |\xi_2|^2)| \lesssim \max\{ L,L_1,L_2\}
\]
If $N \ll N_1$ then the above left hand side must have size
$N N_1$. Thus we have established the following dichotomy:
\begin{equation}
\mbox{ either}\; N \sim N_1 \; \mbox{ or } \;  N N_1  
\lesssim \max\{ L,L_1,L_2\}.
\label{dich}\end{equation}
Then we can use the same argument as in Case
\eref{it:largeN} of the proof of Proposition
\ref{prop:trans_high_mod}. 
 
If $L=\max\{L_1,L_2,L\}$ then the bound \eref{bsize} holds, and
corresponding to the two cases in \eref{dich} we only need to compute
\[
L_1 L_2 \frac{1}{N} \frac{N_1}{A} = L_1 L_2 \frac{1}{N} \lesssim  
L_1^\frac23 L_2^\frac23 L^\frac23 \frac{1}{N}\frac{N}{N_1}
\]
respectively
\[
L_1 L_2 \frac{1}{N} \frac{N_1}{A} = L_1^\frac56 L_2^\frac56 
L^{\frac13}  \frac{1}{N} \lesssim  L_1^\frac56 L_2^\frac56 
L^{\frac56}  \frac{1}{N} \frac{1}{(NN_1)^\frac12} 
\]
both of which are stronger than needed.

On the other hand if $L_1 =  \max\{ L,L_1,L_2\}$ then 
\eref{eq:setc} holds, and we conclude as above taking into account
the two cases in \eref{dich}. The case $L_2 =  \max\{ L,L_1,L_2\}$ is similar.
\end{proof}

The next proposition covers the case of high-low interactions.
\begin{proposition}[high-low interactions]\label{prop:high-low}
 Let  $f,g_1,g_2\in L^2$ be functions with
 $\|f\|_{L^2}=\|g_1\|_{L^2}=\|g_2\|_{L^2}=1$
  such that
  \begin{equation*}
    \supp(f)\subset \mathfrak{P}_{N}\cap \mathfrak{W}^\pm_L ,\quad 
    \supp(g_k)\subset \mathfrak{P}_{N_k} \cap \mathfrak{S}_{L_k}\quad (k=1,2),
  \end{equation*}
  with $1\leq N_1\ll N_2$ or $1\leq N_2\ll N_1$.  Then, for all
  $L,L_1,L_2\geq 1$ we have
  \begin{equation}\label{eq:high-low}
  |I(f,g_1,g_2)|  
    \ls L_{1}^{\frac{5}{12}}L_{2}^{\frac{5}{12}}L^{\frac{5}{12}}N^{-\frac12}
    \min\left\{\frac{N_1}{N_2},\frac{N_2}{N_1}\right\}^\frac{1}{6}
  \end{equation}
\end{proposition}
\begin{proof}
  Assume first that $N_1\ll N_2$. Then, the integral vanishes unless
  $N_2 \sim N$ and
  \begin{equation}\label{eq:res_in}
    \max\{L,L_1,L_2\}\gs ||\xi_1|^2-|\xi_2|^2\pm
 |\xi_1-\xi_2||\gs N_2^2.
  \end{equation}
We consider three cases:

{\it Case 1:} $L=\max\{L,L_1,L_2\}$. Then by the bilinear Strichartz
estimate \eref{eq:str-schr-schr} we have
  \begin{eqnarray*}
    |I(f,g_1,g_2)| 
    \ls 
    \|f\|_{L^2}\|\mathcal{F}^{-1}g_1\overline{\mathcal{F}^{-1}g_2}\|_{L^2}
    \ls   L_1^\frac12
    L_2^\frac12  \left(\frac{N_1}{N_2}\right)^\frac12
  \end{eqnarray*}
  Then the claim follows due to \eref{eq:res_in}.

  {\it Case 2:} $L_1=\max\{L,L_1,L_2\}$. Since $g_1$ is localized in
  frequency in a cube of size $N_1$, by orthogonality the estimate
  reduces to the case when $f$ and $g_2$ are frequency localized in
  cubes of size $N_1$.  Then we use bilinear $L^2$ estimate
  \eref{eq:str-wave-schr-gen} with $d=N_1$ to obtain
  \begin{eqnarray*}
   |I(f,g_1,g_2)|    
    \ls  \|g_1\|_{L^2}\|\mathcal{F}^{-1} f\mathcal{F}^{-1} g_2\|_{L^2}
    \ls    L^\frac12
    L_2^\frac12  \left(\frac{N_1}{N_2}\right)^\frac12
  \end{eqnarray*}
 and conclude again using \eref{eq:res_in}.

 {\it Case 3:} $L_2=\max\{L,L_1,L_2\}$.  On 
  one hand,  by \eref{eq:str-wave-schr} we obtain the bound
  \begin{eqnarray*}
    |I(f,g_1,g_2)|   \ls 
    \|g_2\|_{L^2}\|\overline{\mathcal{F}^{-1}f}\mathcal{F}^{-1}g_1\|_{L^2}
    \ls L^\frac12 L_1^\frac12
  \end{eqnarray*}
  which implies \eref{eq:high-low} if additionally
  $L_1\leq N_1^2$ holds.

  On the other hand, by Young's inequality we have
  \begin{equation*}
    |I(f,g_1,g_2)| 
    \leq  \|g_2\|_{L^2}\|f\|_{L^2_\xi L^1_\tau }\|g_1\|_{L^1_\xi L^2_\tau }
    \ls  L^\frac12 N_1.
  \end{equation*} which, combined with \eref{eq:res_in}, suffices
  in the elliptic regime $L_1>N_1^2$.

  The case $N_1\gg N_2$ follows by the same arguments.
\end{proof}

Finally, we deal with the case where the wave frequency is very small.
\begin{proposition}[Very small wave
  frequency]\label{prop:very_small_wave_freq}
  Let $f,g_1,g_2\in L^2$ with
  $\|f\|_{L^2}=\|g_1\|_{L^2}=\|g_2\|_{L^2}=1$ such that
  \begin{equation*}
    \supp(f)\subset \mathfrak{P}_{N}\cap \mathfrak{W}^\pm_L ,\quad 
    \supp(g_k)\subset \mathfrak{P}_{N_k} \cap \mathfrak{S}_{L_k}\quad (k=1,2),
  \end{equation*}
  and assume that $N\ls 1$. Then,
  \begin{equation}\label{eq:very_small_wave_freq}
  |I(f,g_1,g_2)|  
    \ls L^{\frac13}L_{1}^{\frac{1}{3}}L_{2}^{\frac{1}{3}}.
  \end{equation}
\end{proposition}
\begin{proof}
  Depending on which of $L,L_1,L_2$ is maximal we apply the bilinear
  Strichartz refinements \eref{eq:str-schr-schr} or
  \eref{eq:str-wave-schr} and the result follows.
\end{proof}

We are ready to provide a proof of our main trilinear estimate
\eref{eq:trilinear}.

\begin{proof}[Proof of Proposition \ref{prop:trilinear}]
  By definition of the norms it is enough to consider functions with
  non-negative Fourier transform. We dyadically decompose
  \begin{equation*}
    u_i=\sum_{N_i,L_i\geq 1} S_{L_i}P_{N_i}u_i \; ,\quad v=\sum_{N,L\geq 1 } W^\pm_{L}P_{N} v.
  \end{equation*}
  Setting $g^{L_i,N_i}_i=\mathcal{F}S_{L_i}P_{N_i}u_i$ and
  $f^{L,N}=\mathcal{F}W^\pm_{L}P_{N} v$, we observe
  \begin{eqnarray*}
   I( \mathcal{F}v,
    \mathcal{F}u_1, \mathcal{F}u_2)
    = {} \sum_{N,N_1,N_2\geq 1} \sum_{L,L_1,L_2\geq 1} I(
    f^{L,N}, g_1^{L_1,N_1},
    g_2^{L_2,N_2}).
  \end{eqnarray*}

  \textit{Case 1:} high-high-low interactions, i.e. $N_1\sim N_2\gs N\geq 2^{10}$.\\
  We fix $M=2^{-4}N_1$ and use the decomposition \eref{eq:dyadic_angular} to write
  \begin{eqnarray*}
 \fl   I(f^{L,N},g_1^{L_1,N_1},g_2^{L_2,N_2})&={}&\sum_{0\leq
        j_1,j_2\leq M-1 \atop |j_1-j_2|\leq 16}
    I(f^{L,N},g_1^{L_1,N_1,M,j_1},g_2^{L_2,N_2,M,j_2})\\
    &&+\sum_{64\leq A \leq M}\sum_{0\leq j_1,j_2\leq
        A-1 \atop 16\leq |j_1-j_2|\leq 32} 
    I(f^{L,N},g_1^{L_1,N_1,A,j_1},g_2^{L_2,N_2,A,j_2})
  \end{eqnarray*}
  where
  $g_i^{L_i,N_i,A,j_i}=g_i^{L_i,N_i}|_{\mathfrak{Q}_{j_i}^{A}}$. We
  apply Proposition \ref{prop:paral_hh} to the first term and 
use Cauchy-Schwarz to obtain
  \begin{eqnarray*}
\fl & \sum_{0\leq j_1,j_2\leq M-1 \atop |j_1-j_2|\leq 16}
    I(f^{L,N},g_1^{L_1,N_1,M,j_1},g_2^{L_2,N_2,M,j_2})\\
\ls{}&
    \frac{(LL_1L_2)^{\frac{5}{12}}}{N^{\frac12}}\left(\frac{N}{N_1}\right)^{\frac14}\|f^{L,N}\|_{L^2}
    \sum_{0\leq j_1,j_2\leq M-1 \atop |j_1-j_2|\leq 16}\|g_1^{L_1,N_1,M,j_1}\|_{L^2}\|g_2^{L_2,N_2,M,j_2}\|_{L^2}\\
    \ls{}&
    \frac{(LL_1L_2)^{\frac{5}{12}}}{N^{\frac12}}\left(\frac{N}{N_1}\right)^{\frac14}\|f^{L,N}\|_{L^2}\|g_1^{L_1,N_1}\|_{L^2}\|g_2^{L_2,N_2}\|_{L^2}.
  \end{eqnarray*}

  Concerning the second term, we split the sum with respect to $A$ into
  two parts according to the quantity
\[
\alpha:=2^{-4}\min\left\{\left(\frac{N_1}{N}\right)^{\frac12}N_1\max\{L,L_1,L_2\}^{-\frac12},N_1\right\}.
\]

For the part where $64\leq A\leq \alpha$ we apply Proposition
\ref{prop:trans_low_mod} and obtain
\begin{eqnarray*}
\fl S_1&:=\sum_{64\leq A \leq \alpha}\sum_{0\leq
      j_1,j_2\leq A-1 \atop 16\leq |j_1-j_2|\leq 32}
  I(f^{L,N},g_1^{L_1,N_1,A,j_1},g_2^{L_2,N_2,A,j_2})\\
 \fl &\ls \left(\frac{LL_1L_2}{N_1}\right)^{\frac{1}{2}}
  \|f^{L,N}\|_{L^2}\!\! \sum_{64\leq A \leq
    \alpha}\frac{A^{\frac12}}{N_1^{\frac12}}
  \sum_{0\leq j_1,j_2\leq A-1 \atop 16\leq |j_1-j_2|\leq
      32}  \|g_1^{L_1,N_1,A,j_1}\|_{L^2}\|g_2^{L_2,N_2,A,j_2}\|_{L^2}.
\end{eqnarray*}
Then, we use Cauchy-Schwarz with respect to $j_1,j_2$
\begin{eqnarray*}
\fl  S_1 &\ls{}\left(\frac{LL_1L_2}{N_1}\right)^{\frac{1}{2}}
  \|f^{L,N}\|_{L^2}\|g_1^{L_1,N_1}\|_{L^2}\|g_2^{L_2,N_2}\|_{L^2}\sum_{64\leq A \leq \alpha}\frac{A^{\frac12}}{N_1^{\frac12}}\\
\fl  &\ls{}(LL_1L_2)^{\frac{5}{12}}N^{-\frac12}\left(\frac{N}{N_1}\right)^{\frac14}\|f^{L,N}\|_{L^2}\|g_1^{L_1,N_1}\|_{L^2}\|g_2^{L_2,N_2}\|_{L^2},
\end{eqnarray*}
due to the property of the dyadic sum $\sum_{64\leq A\leq
  \alpha}A^{\frac12}\ls \alpha^\frac12$.

For the part where $\alpha\leq A\leq N_1$ we use Proposition
\ref{prop:trans_high_mod} and obtain
\begin{eqnarray*}
 \fl S_2&:=\sum_{\alpha\leq A \leq M}\sum_{0\leq
      j_1,j_2\leq A-1 \atop 16\leq |j_1-j_2|\leq 32}
  I(f^{L,N},g_1^{L_1,N_1,A,j_1},g_2^{L_2,N_2,A,j_2})\\
  \fl &\ls\frac{L^\frac12L_1^\frac12 L_2^\frac12
    \|f^{L,N}\|_{L^2}}{\max\{L_1,L_2,L\}^\frac12
    N^\frac12}\sum_{\alpha \leq A \leq
    M}\frac{N_1^\frac12}{A^{\frac12}} \sum_{0\leq
      j_1,j_2\leq A-1 \atop 16\leq |j_1-j_2|\leq 32}
  \|g_1^{L_1,N_1,A,j_1}\|_{L^2}\|g_2^{L_2,N_2,A,j_2}\|_{L^2}.
\end{eqnarray*}
As above, we use Cauchy-Schwarz with respect to $j_1,j_2$ and obtain
\begin{eqnarray*}
\fl  S_2&\ls{}\frac{L^\frac12L_1^\frac12 L_2^\frac12
    \|f^{L,N}\|_{L^2}\|g_1^{L_1,N_1}\|_{L^2}\|g_2^{L_2,N_2}\|_{L^2}}{\max\{L_1,L_2,L\}^\frac12
    N^\frac12}
  \sum_{\alpha \leq A \leq M}\frac{N_1^{\frac12}}{A^{\frac12}}\\
\fl  &\ls{}\frac{(LL_1L_2)^{\frac{5}{12}}}{N^{1/2}}\left(\frac{N}{N_1}\right)^{\frac14}\|f^{L,N}\|_{L^2}\|g_1^{L_1,N_1}\|_{L^2}\|g_2^{L_2,N_2}\|_{L^2},
\end{eqnarray*}
because of $\sum_{\alpha \leq A\leq M} A^{-\frac12}\ls
\alpha^{-\frac12}$.

{\it Case 2}: very small wave frequency, i.e. $N\ls 1$. In this case,
either $N_1\sim N_2$ or $N,N_1,N_2\ls 1$ and we apply Proposition
\ref{prop:very_small_wave_freq} and arrive at the bound
\eref{eq:very_small_wave_freq}

{\it Case 3}: high-low interactions, i.e. $N_1\ll N_2$ or $N_1\gg
N_2$. We apply Proposition \ref{prop:high-low} and obtain the bound
\eref{eq:high-low}.

To summarize, we obtain in any case the weakest of all three bounds,
namely
\begin{eqnarray*}
\fl  &I(f^{L,N},g_1^{L_1,N_1},g_2^{L_2,N_2})\\
\fl  \ls{}&
  (LL_1L_2)^{\frac{5}{12}}\min\left\{\frac{N}{N_1},\frac{N_1}{N_2},\frac{N_2}{N_1}\right\}^{\frac16}
  \frac{\|f^{L,N}\|_{L^2}}{N^{\frac12}}\|g_1^{L_1,N_1}\|_{L^2}\|g_2^{L_2,N_2}\|_{L^2},
\end{eqnarray*}
which we dyadically sum with respect to $L,L_1,L_2\geq 1$. Then, we
use that for non-vanishing contributions we must have $N\ls N_1 \sim
N_2$ or $N_1\ls N\sim N_2$ or $N_2\ls N$ and the prefactor enables us
to control the sum by the corresponding dyadic $\ell^2$-norms.
\end{proof}

\section{Linear estimates and the proof of Theorem \ref{main}}\label{sect:proof}
\noindent
Before we prove Theorem \ref{main} we present some linear estimates which are well-known at least in the case of standard Bourgain spaces, see e.g. \cite[Section 2]{ginibre_cauchy_1997}.

We define the 1d inhomogeneous Besov norms
\[
\|g\|_{B^b_{2,1}}=\sum_{L\geq 1}L^b \|P_L g\|_{L^2},\quad
\|g\|_{B^b_{2,\infty}}=\sup_{L\geq 1} L^b \|P_L g\|_{L^2}.
\]
For $0<T\leq 1$ we define a smooth cutoff function for the interval $[0,T]$ as $\psi_T(t)=\psi(t/T)$ and we define the Fourier localization operator $P_{\leq T^{-1}}:=\sum_{1\leq L\leq T^{-1}}P_{L}$, cp. Section \ref{sect:not}.
\begin{lemma}\label{lem:norm_equiv}
Let $0<b \leq \frac12$. For all $g \in \mathcal{S}(\R)$ and $T\in (0,1]$ we have
  \begin{equation}\label{eq:norm_equiv}
\|g\psi_T\|_{B^b_{2,1}}\sim T^{-b} \|P_{\leq T^{-1}}(g \psi_T)\|_{L^2}+\sum_{L >T^{-1}} L^{b}\|P_L (g\psi_T)\|_{L^2},
  \end{equation}
where the implicit constants are independent of $T$ and $g$.
\end{lemma}
\begin{proof}
On the one hand we have
  \begin{eqnarray*}
    \sum_{1\leq L \leq T^{-1}} L^b \|P_L(g\psi_T)\|_{L^2}&\leq &
    2  \left(\sum_{1\leq L\leq T^{-1}} L^{2b}\right)^{\frac12} \|P_{\leq T^{-1}}(g\psi_T)\|_{L^2}\\
   &\ls&  T^{-b} \|P_{\leq T^{-1}}(g\psi_T)\|_{L^2},
 \end{eqnarray*}
 and on the other hand
\[
T^{-b} \|P_{\leq T^{-1}}(g\psi_T)\|_{L^2}\leq  T^{-b} \|g\psi_T\|_{L^2}
\leq \|g \psi_T\|_{L^{\frac{2}{1-2b}}}
\ls \|g \psi_T\|_{B^b_{2,1}}
\]
where we have used the embedding $B^{b}_{2,1}\subset L^{\frac{2}{1-2b}}$ in the last step.
\end{proof}
In the following, let
$X_{s,b,p}(T)$ denote either $X^S_{s,b,p}(T)$ or
$X^{W\pm}_{s,b,p}(T)$.
\begin{proposition}\label{prop:emb}
  Let $s,b \in \R$, $0< b <\frac12$. There
  exists a constant $C>0$ such that for all $T \in (0,1]$ the estimate
  \begin{equation}
    \|f\|_{X_{s,b,1}(T)}\leq C T^{\frac12-b} \|f\|_{X_{s,\frac12,1}(T)}\label{eq:emb_bp}
  \end{equation}
  holds for all $f \in X_{s,\frac12,1}(T)$.  Moreover, the embedding  $X_{s,\frac12,1}(T) \subset C([0,T];H^s)$ is continuous, i.e. there
  exists a constant $C>0$ such that for all $T \in (0,1]$ it holds
  \begin{equation}\label{eq:cont_emb}
    \sup_{0\leq t \leq T}\|f(t)\|_{H^s}\leq C \|f\|_{X_{s,\frac12,1}(T)}
  \end{equation}
  for all $f \in X_{s,\frac12,1}(T)$. 
\end{proposition}
\begin{proof}
  We show \eref{eq:emb_bp} first. By the definition of the restriction norm it suffices to prove
  \[
  \|f\psi_T\|_{X_{s,b,1}}\leq C T^{\frac12-b} \|f\|_{X_{s,\frac12,1}}
  \]
  for all $f \in \mathcal{S}(\R^n \times \R)$.
  After conjugating $f$ with the linear group the claim is reduced to the estimate
  \[
     \|g\psi_T \|_{B^b_{2,1}}\leq C T^{\frac12-b} \|g\|_{B^{\frac12}_{2,1}}
  \]
  for $g \in \mathcal{S}(\R)$.
Then, with $g_T(t)=g(Tt)$ we use \eref{eq:norm_equiv} and obtain
\begin{eqnarray*}
\|g \psi_T\|_{B^b_{2,1}}&\ls& T^{-b} \|P_{\leq T^{-1}}(g \psi_T)\|_{L^2}+\sum_{L >T^{-1}} L^{b}\|P_L (g\psi_T)\|_{L^2}\\&\ls&
T^{\frac12-b} \sum_{L \geq 1} L^{b}\|P_L (g_T\psi)\|_{L^2}\\
&\ls& T^{\frac12-b}(\|g_T \psi\|_{L^2}+\|g_T \psi\|_{\dot{H}^\frac12})
\end{eqnarray*}
by rescaling. Obviously,
\[
\|g_T \psi\|_{L^2}\leq \|g_T\|_{L^\infty}\ls \|g\|_{B^\frac12_{2,1}}
\]
and by the 1d Sobolev Multiplication Theorem
\[
\|g_T \psi\|_{\dot{H}^{\frac12}}\ls \|g_T\|_{\dot{H}^{\frac12}}+\|g_T\|_{L^\infty}\ls \|g\|_{B^{\frac12}_{2,1}}.
\]

The second claim, including formula \eref{eq:cont_emb}, follows from the continuous embedding
$B^{\frac12}_{2,1}\subset C(\R;\R)$.
\end{proof}

For $f \in \mathcal{S}(\R^2\times \R)$ and $t\in \R$ let
\begin{eqnarray}
  \mathcal{I}^S(f)(t)&:=\int_0^te^{i(t-s)\Delta}f(s)ds,\label{eq:duhamel_schr}\\
  \mathcal{I}^{W+}(f)(t)&:=\int_0^te^{-i(t-s)\langle \nabla\rangle}f(s)ds.\label{eq:duhamel_wave}
\end{eqnarray}

The following Proposition corresponds to \cite[Lemma 2.1]{ginibre_cauchy_1997}.
\begin{proposition} \label{lest} Let $s \in \R$. There exists $C>0$ such that for all $0<T\leq 1$ and $\phi \in H^s$ the estimates
  \begin{eqnarray}
    \|e^{it\Delta}\phi \|_{X^S_{s,\frac12,1}(T)} &\leq  C  \|\phi\|_{H^s} \label{eq:hom_est_s},\\
    \|e^{-it\langle \nabla\rangle}\phi \|_{X^{W+}_{s,\frac12,1}(T)} &\leq C  \|\phi\|_{H^s},\label{eq:hom_est_w}
  \end{eqnarray}
 are true, and moreover the estimates
  \begin{eqnarray}
    \| \mathcal{I}^S(f)\|_{X^S_{s,\frac12,1}(T)} &\leq  C  T^{\frac{1}{12}} \| f  \|_{X^S_{s,-\frac5{12},\infty}(T)} ,\label{eq:duhamel_schr_est}\\
    \| \mathcal{I}^{W+}(f)\|_{X^{W+}_{s,\frac12,1}(T)} & \leq C  T^{\frac{1}{12}} \| f
    \|_{X^{W+}_{s,-\frac5{12},\infty}(T)}, \label{eq:duhamel_wave_est}
  \end{eqnarray}
  are true for all sufficiently smooth $f$. Therefore, $\mathcal{I}^S$ and $\mathcal{I}^{W+}$ can be
  extended to continuous linear operators on these spaces, which
  satisfy the same bounds.
\end{proposition}
\begin{proof}
  We use the notation as in Proposition \ref{prop:emb} above.

  First, \eref{eq:hom_est_s} and \eref{eq:hom_est_w} are proved as in \cite[equation (2.19)]{ginibre_cauchy_1997}
  upon replacing the Sobolev space $H^{\frac12}_t$ by the Besov space $B^{\frac12}_{2,1}$.

  Second, by choosing appropriate extensions and conjugating with the linear group,
  the estimates \eref{eq:duhamel_schr_est}
  and \eref{eq:duhamel_wave_est}
  easily reduce to the estimate
  \[
  \|\psi_T \mathcal{I}(g)\|_{B^{\frac12}_{2,1}}\leq C T^{\frac1{12}}\|g\|_{B^{-\frac{5}{12}}_{2,\infty}}
  \]
  for all $g \in \mathcal{S}(\R)$, $T \in (0,1]$, where $\mathcal{I}(g)=\int_0^t g(t')dt'$. With $g_T(t)=g(Tt)$ we calculate
  \[
  (\psi_TI(g))(Tt)=T \psi(t) I(g_T)(t).
  \]
  Now, \eref{eq:norm_equiv} and rescaling yields
  \[
  \|\psi_T \mathcal{I}(g)\|_{B^{\frac12}_{2,1}}\leq C T \|\psi I(g_T)\|_{B^{\frac12}_{2,1}}.
  \]
  From estimate \cite[formula (2.24)]{ginibre_cauchy_1997} with $T=1$ and trivial embeddings we deduce
  \[
  \|\psi I(g_T)\|_{B^{\frac12}_{2,1}} \leq C \|g_T\|_{B^{-\frac{5}{12}}_{2,\infty}}.
  \]
  Finally, rescaling shows that
  \[
  \|g_T\|_{B^{-\frac{5}{12}}_{2,\infty}}\leq C T^{-\frac{11}{12}}\|g\|_{B^{-\frac{5}{12}}_{2,\infty}}
  \]
  for all $0<T\leq 1$, which concludes the proof.
\end{proof}

\begin{definition}\label{def:sol}
  We call $(u,v)\in X^{S}_{s,\frac12,1}(T)\times
  X^{W+}_{s',\frac12,1}(T)$ a solution of \eref{eq:zs-red} with
  initial data $(u_0,v_0)\in H^{s}\times H^{s'}$, if it solves
  \begin{equation}\label{eq:sol} \left(\begin{array}{c}u(t)\\v(t)\end{array}\right)=\left(\begin{array}{c}e^{it\Delta}u_0\\e^{-it\langle\nabla\rangle}v_0\end{array}\right)-i\left(\begin{array}{c}\mathcal{I}^{S}(2\Re(v)u)(t)\\\mathcal{I}^{W+}(-\frac{\Delta}{\langle\nabla\rangle}|u|^2-\frac{1}{\langle \nabla \rangle}\Re v)(t)
      \end{array}\right)
    \end{equation}
  for all $t \in [0,T]$.
\end{definition}

Now we are ready to proceed with the proof of our main result.

\begin{proof}[Proof of Theorem \ref{main}]
  Let $R=\| u_0 \|_{L^2}+\|v_0 \|_{H^{-\frac12}}$. Since the time 
  of existence claimed in Theorem \ref{main} is smaller than $1$
  it is enough  to discuss only the case $1 \ls R$.  

  The estimates \eref{eq:duhamel_schr_est} and \eref{eq:sws1}, \eref{eq:sws2}
  yield
  \begin{eqnarray*}
    \| \mathcal{I}^{S}(2\Re(v)u)\|_{X^S_{0,\frac12,1}(T)} 
    &\ls& T^\frac{1}{12} ( \| uv\|_{X^S_{0,-\frac{5}{12},\infty}(T)} + \| u \bar{v}\|_{X^S_{0,-\frac{5}{12},\infty}(T)}) \\
    &\ls& T^\frac{1}{12} \| u \|_{X^S_{0,\frac{5}{12},1}(T)} \|v\|_{X^{W+}_{-\frac12,\frac{5}{12},1}(T)},
  \end{eqnarray*}
  and \eref{eq:emb_bp} implies
  \begin{equation}\label{nonlschr}
    \| \mathcal{I}^{S}(2\Re(v)u)\|_{X^S_{0,\frac12,1}(T)} \ls T^\frac{1}{4} \| u \|_{X^S_{0,\frac12,1}(T)} \| v\|_{X^{W+}_{-\frac12,\frac12,1}(T)}.
  \end{equation}
  In a similar manner, using \eref{eq:wss}, we estimate
  \begin{eqnarray*}
    \left\|\mathcal{I}^{W+}(\frac{\Delta}{\langle\nabla\rangle}|u|^2)\right\|_{X^{W+}_{-\frac12,\frac12,1}(T)}
&\ls & T^{\frac{1}{12}}\left\|\frac{\Delta}{\langle\nabla\rangle}|u|^2\right\|_{X^{W+}_{-\frac{1}{2},-\frac{5}{12},\infty}(T)} \\
    &\ls& T^\frac{1}{12} \| u \|^2_{X^S_{0,\frac{5}{12},1}(T)}
  \end{eqnarray*}
  and obtain
  \begin{equation} \label{nonlwave1}
  \left\|\mathcal{I}^{W+}(\frac{\Delta}{\langle\nabla\rangle}|u|^2)\right\|_{X^{W+}_{-\frac12,\frac12,1}(T)}
    \ls T^\frac14 \|u_1\|_{X^S_{0,\frac12,1}(T)}
    \|u_2\|_{X^S_{0,\frac12,1}(T)}.
  \end{equation}
  Additionally we obtain
        \begin{equation} \label{nonlwave2}
        \eqalign{
    \|\mathcal{I}^{W+}(\langle \nabla \rangle^{-1} \Re v)\|_{X^{W+}_{-\frac12,\frac12,1}(T)}
    &\ls \|\langle \nabla \rangle^{-1} \Re v \|_{L^2([0,T] \times \R^2)} \\
    &\ls T^\frac12 \| v \|_{L^{\infty}_t H^{-\frac12}_x} \ls T^\frac12 \| v \|_{X_{-\frac12,\frac12,1}^{W+}},}
  \end{equation}
  which easily follows from \eref{eq:duhamel_wave_est}.
  The analoguos estimates for differences can be shown by the same arguments.
  Using these nonlinear estimates and the linear estimates in
  Proposition \ref{lest}, a standard iteration argument constructs a
  unique solution
  \[(u,v) \in B_{X^S_{0,\frac12,1}(T)}(0,C \|u_0\|_{L^2}) \times B_{X^{W+}_{-\frac12,\frac12,1}(T)}(0, C \|v_0\|_{H^{-\frac12}})\]
  for \eref{eq:sol}, provided that $T \sim R^{-4}$. In addition, one can show local Lipschitz continuity of the
  induced map $(u_0,v_0)\mapsto (u,v)$.
  
  Next we seek to boost the time of existence based on the technique
  described in \cite{colliander_low_2008}. This is possible due to the
  $L^2$ norm conservation for $u$ and to the fact that the
  nonlinearity for $v$ depends only on $u$. We claim that the time of
  existence can be improved to $T \sim \min\{R^{-2} \| u_0 \|_{L^2}^{-2},1\}$.

  Without restricting the generality of the argument we can assume
  that $\| v_0 \|_{H^{-\frac12}} \geq \| u_0 \|_{L^2}$. Then by the
  above argument we are able to construct solutions on the time interval
  $\delta \sim \| v_0 \|_{H^{-\frac12}}^{-4}$.

  On the other hand, using \eref{eq:sol}, \eref{nonlwave1},  \eref{nonlwave2} and that
  $e^{-it\la \nabla \ra}$ is unitary we obtain
  \begin{eqnarray*}
  \fl  \| v \|_{L^\infty_t H_x^{-\frac12}([0,\delta] \times \R^2)} &
    \leq \| v_0 \|_{H^{-\frac12}(\R^2)} + C \delta^\frac14 \| u \|^2_{X^S_{0,\frac12,1}(\delta)} + \| v \|_{L^1_t H_x^{-\frac12}([0,\delta] \times \R^2)}\\
\fl    &\leq \| v_0 \|_{H^{-\frac12}(\R^2)} + C \delta^\frac14 \| u_0
    \|^2_{L^2} +  \delta \| v \|_{L^\infty_t H_x^{-\frac12}([0,\delta] \times \R^2)}.
  \end{eqnarray*}
  This allows us to keep reiterating the problem on intervals
  $[j\delta,(j+1)\delta]$ for $j=0,1,\ldots,m$ until we double the
  size of the wave data, i.e. up to the first time when
  $\|v(t_0)\|_{H^{-\frac12}}=2\|v_0\|_{H^{-\frac12}}$ (after this time
  the value of $\delta$ has to be adjusted). After $m$ iterations we obtain
  \[\fl
  \| v \|_{L^\infty_t H_x^{-\frac12}([0,m\delta] \times \R^2)} \leq \| v_0 \|_{H^{-\frac12}(\R^2)} + C m \delta^\frac14 \| u_0
    \|^2_{L^2} +  m \delta \| v \|_{L^\infty_t H_x^{-\frac12}([0,m\delta] \times \R^2)}.
  \]
  A direct computation
  gives $m \sim \min{(\| v_0 \|_{H^{-\frac12}}\delta^{-1/4} \| u_0
  \|^{-2}_{L^2},\delta^{-1})}$ and this improves the time of existence for solutions
  to
  \[
  m \delta \sim \min{(\frac{R}{C R^{-1} \| u_0 \|^2_{L^2}} R^{-4},1)} \sim \min{(R^{-2}\| u_0 \|^{-2}_{L^2},1)}.
  \]
  Therefore we are able to improve the life-span of
  solution to a time $T \sim \min{(R^{-2} \| u_0 \|^{-2}_{L^2},1)}$ which implies the claim
  in Theorem \ref{main}.

  Then a standard argument also establishes the uniqueness of
  solutions in $X^S_{0,\frac12,1}(T) \times X^{W+}_{-\frac12,\frac12,1}(T)$ and the Lipschitz dependence with
  respect to the initial data.
\end{proof}

\section{Counterexamples}\label{sect:counter}
\noindent
We first show that the time of existence provided in Theorem \ref{main} is
optimal up to the multiplicative constant.
\begin{proof}[Proof of Theorem \ref{thm:blowup}]
Fix $r>\|Q\|_{L^2}$. There exists $\omega\gg 1$ such that the Glangetas--Merle \cite{glangetas_existence_1994,glangetas_concentration_1994} solution $P_\omega$, see \eref{eq:blowup_sol}, satisfies $\|P_{\omega}\|_{L^2}<r$. We fix such $\omega\gg 1$ and calculate for the corresponding solution \eref{eq:blowup_sol}
\begin{equation*}
\|u(t)\|_{L^2}=\|P_\omega\|_{L^2}<r,
\end{equation*}
and
\begin{eqnarray*}
\|n(t)\|_{H^{-\frac12}}+\|\partial_t n(t)\|_{H^{-\frac32}}\sim |T-t|^{-\frac12}.
\end{eqnarray*}
Theorem \ref{thm:blowup} follows.
\end{proof}

Next, we show that our multilinear estimates in Theorem \ref{thm:tri}
are sharp.  We follow the approach which has been pioneered by Bourgain \cite{bourgain_periodic_1997} to show non-smoothness of the flow map. We also refer the reader to \cite{holmer_local_2007} where related counterexamples in the 1d case have been constructed.

In order to avoid unnecessary technicalities, we write
$X^S_{k,b}$ to denote $X^S_{k,b,2}$ and $X^{W\pm}_{\ell,b}$ to denote
$X^{W\pm}_{\ell,b,2}$ and provide counterexamples for this scale of
norms.  We remark that the arguments remain valid for any choice of
$1\leq p\leq \infty$ instead of $2$.  The reason is that the norms
$X^S_{k,b,p}$ for distinct $p$ are equivalent up to logarithms of the
size of the modulation (same for $X^{W\pm}_{k,b,p}$), but our
counterexamples will always involve powers of the modulation.

Moreover, in Proposition \ref{prop:c1} we restrict the exposition to
the case of the $X^{W+}_{k,b}$ space, i.e. the sharpness of
\eref{eq:sws2}; the case $X^{W+}_{k,b}$, i.e. the sharpness of
\eref{eq:sws1}, follows by the same argument up to obvious
modifications.
\begin{proposition}\label{prop:c1}
  The inequality
$$\|u v\|_{X^S_{k,-b'}} \lesssim \|v\|_{X^{W+}_{\ell,b_1}}\|u\|_{X^{S}_{k,b_2}}$$
is false in either of the following two situations:
\begin{enumerate}
\item if $\ell<-\frac12$, for any $b'$, $b_1$, and $b_2$,
\item if $\ell=-\frac12$ and $b'+b_1+b_2< \frac54$.
\end{enumerate}
\end{proposition}

This follows from applying Lemma \ref{L:counter1} with $\sigma=-1$ to
establish the first claim, and any $\sigma$ such that $-1<\sigma<0$ to
establish the second claim.

\begin{proposition}\label{prop:c2}
  The inequality
$$\left\|\frac{\Delta}{\langle \nabla\rangle} (u \bar w)\right\|_{X^{W+}_{\ell,-b'}} \lesssim \|u\|_{X^S_{k,b_1}}\|w\|_{X^S_{k,b_2}}$$
is false in either of the following two situations:
\begin{enumerate}
\item if $\ell-2k+\frac12>0$ for any $b'$, $b_1$, and $b_2$,
\item if $\ell-2k+\frac12=0$ and $b'+b_1+b_2< \frac54$
\end{enumerate}
\end{proposition}

This follows from applying Lemma \ref{L:counter2} with $\sigma=-1$ to
establish the first claim, and any $\sigma$ such that $-1<\sigma<0$ to
establish the second claim.

\begin{lemma}
  \label{L:counter1}
  For each $N\gg 1$, there exist $v_N$ and $u_N$ such that
$$\frac{\|v_N u_N \|_{X^S_{k,-b'}}}{\|v_N\|_{X^{W+}_{\ell,b_1}} \|u_N\|_{X^S_{k,b_2}}} \gtrsim N^{-\ell-\frac12 + (1+\sigma)[ \frac54-(b'+b_1+b_2)]}$$
for all $k,\ell \in \mathbb{R}$ and $ b', b_1, b_2 \geq 0$, and any
$-1\leq \sigma <0$, with the implicit constant independent of all of
$k,\ell, b', b_1, b_2, \sigma$, and $N$.
\end{lemma}
\begin{proof}
  Denote $\xi=(\xi_1,\xi_2)$ (i.e. $\xi_j$ now denotes the $j$th
  component of $\xi$).  Let $\hat v = \chi_E$, where $E$ is the
  rectangle centered at $(\xi_1,\xi_2,\tau)=(2N+1,0,-2N-1)$ and width
  $N^{\sigma}\times N^{\frac12(1+\sigma)}\times N^{1+\sigma}$, so that
  on $E$, we have $|\tau+|\xi|| \lesssim N^{1+\sigma}$.  Let $\hat u
  =\chi_F$, where $F$ is the rectangle centered at
  $(\xi_1,\xi_2,\tau)= (-N,0,-N^2)$ and width $N^{\sigma}\times
  N^{\frac12(1+\sigma)} \times N^{1+\sigma}$, so that on $F$, we have
  $|\tau+|\xi|^2|\lesssim N^{1+\sigma}$.  Then $\widehat{vu} \gs
  N^{\frac32+\frac52\sigma} \chi_G$, where $G$ is a rectangle centered
  at $(\xi_1,\xi_2,\tau) = (N+1,0,-(N+1)^2)$ and width
  $N^{\sigma}\times N^{\frac12(1+\sigma)} \times N^{1+\sigma}$. Note
  that on $G$, we have $|\tau+|\xi|^2| \lesssim N^{1+\sigma}$.
  Then
  \begin{eqnarray*}
    \|vu\|_{X^S_{k,-b'}} &\gtrsim N^{\frac32+\frac52\sigma} N^kN^{-(1+\sigma)b'} \|\chi_G\|_{L^2} \\
    &= N^{\frac32+\frac52\sigma}
    N^kN^{-(1+\sigma)b'}N^{\frac12(\frac32+\frac52\sigma)}\; ,
  \end{eqnarray*}
  and
$$\|v\|_{X^{W+}_{\ell,b_1}} \lesssim N^\ell N^{(1+\sigma)b_1}\|\chi_E\|_{L^2} = N^\ell N^{(1+\sigma)b_1}N^{\frac12(\frac32+\frac52\sigma)} \,,$$
$$\|u\|_{X^S_{k,b_2}} \lesssim N^k N^{(1+\sigma)b_2} \|\chi_F\|_{L^2} = N^k N^{(1+\sigma)b_2} N^{\frac12(\frac32+\frac52\sigma)} \,,$$
which proves the claim.
\end{proof}

\begin{lemma}
  \label{L:counter2}
  For each $N\gg 1$, there exists $u_N$ and $w_N$ such that
$$\frac{\|\frac{\Delta}{\langle \nabla\rangle} (u_N \bar w_N)\|_{X^{W+}_{\ell,-b'}}}{\|u_N\|_{X^S_{k,b_1}} \|w_N\|_{X^S_{k,b_2}}} \gtrsim N^{\ell-2k+\frac12+(1+\sigma)[\frac54-(b'+b_1+b_2)]}$$
for all $k,\ell \in \mathbb{R}$, any $b', b_1, b_2\geq 0$, and any
$-1\leq \sigma <0$, with the implicit constant independent of all of
$k,\ell, b', b_1, b_2, \sigma$, and $N$.
\end{lemma}
\begin{proof}
  Let $\hat u_N = \chi_E$, where $E$ is the rectangle centered at
  $(\xi_1,\xi_2,\tau)=(N+1,0,-(N+1)^2)$ with width $N^{\sigma}\times
  N^{\frac12(1+\sigma)}\times N^{1+\sigma}$, so that on $E$, the
  quantity $|\tau+|\xi|^2 |\leq N^{1+\sigma}$.  Let $\hat w_N=\chi_F$,
  where $F$ is the rectangle centered at
  $(\xi_1,\xi_2,\tau)=(-N,0,-N^2)$ with width $N^{\sigma}\times
  N^{\frac12(1+\sigma)}\times N^{1+\sigma}$, so that on $F$, the
  quantity $|\tau+|\xi|^2| \leq N^{1+\sigma}$.  Then $\widehat{u_N \bar
    w_N} \gs N^{\frac32+\frac52\sigma}\chi_G$, where $G$ is the
  rectangle centered at $(2N+1,0,-2N-1)$ and width $N^{\sigma}\times
  N^{\frac12(1+\sigma)}\times N^{1+\sigma}$ so that on $G$, the
  quantity $|\tau+|\xi|| \leq N^{1+\sigma}$.  Thus,
  \begin{eqnarray*}
    \left\| \frac{\Delta}{\langle \nabla\rangle} ( u_N \bar w_N) \right\|_{X^{W+}_{\ell,-b'}}
    &\gtrsim N^{\frac32+\frac52\sigma} N^{\ell+1} N^{-(1+\sigma)b'}\|\chi_G\|_{L^2} \\
    &= N^{\frac32+\frac52\sigma} N^{\ell+1} N^{-(1+\sigma)b'}
    N^{\frac12(\frac32+\frac52\sigma)}\; ,
  \end{eqnarray*}
  and
$$\|u_N\|_{X^S_{k,b_1}} \lesssim N^k N^{(1+\sigma)b_1} \|\chi_E\|_{L^2} = N^k N^{(1+\sigma)b_1}N^{\frac12(\frac32+\frac52\sigma)} \,,$$
$$\|w_N\|_{X^S_{k,b_2}} \lesssim N^k N^{(1+\sigma)b_2} \|\chi_F\|_{L^2} =  N^k N^{(1+\sigma)b_2}N^{\frac12(\frac32+\frac52\sigma)} \,,$$
which proves the claim.
\end{proof}

\begin{remark} Alternatively, the optimality of our choice of $b_1=b_2=b_3=\frac{5}{12}$ can be seen by an indirect argument: If it was possible to choose smaller $b$'s, we would be able to improve the time of existence by the iterative argument given in Section \ref{sect:proof} above and would obtain a contradiction to the blow-up of the Glangetas--Merle solutions constructed in \cite{glangetas_concentration_1994,glangetas_existence_1994}.
\end{remark}

The following proposition is based on a variant of the example from the
proof of Proposition \ref{prop:c1} and contains a slightly stronger conclusion.
\begin{proposition}\label{prop:unbounded_der1} Fix $0<T\leq 1$. For all $N\gg T^{-1}$
there exists $u_N\in H^k_x$ and $v_N\in H^{\ell}_x$ such that
\begin{equation*}
\fl \sup_{|t|\leq T}\left\|\int_0^te^{i(t-t')\Delta}\left(e^{it'\Delta}u_N \Re\left(e^{-it'\langle \nabla\rangle}v_N\right)\right)dt'\right\|_{H^{k}_x}\gs \frac{\|u_N\|_{H^k_x}\|v_N\|_{H^\ell_x}}{N^{\ell+\frac12}}\; ,
\end{equation*}
where the constant is independent of $N$.
\end{proposition}
\begin{proof}
Set $\hat u_N:=\chi_{A}$, where $A$ is the rectangle
where $\xi=(\xi_1,\xi_2)$ satisfies
$$ -N-N^{-1}\leq \xi_{1}\leq -N+N^{-1} \;\mbox{ and }\; -1\leq \xi_2\leq 1,$$
such that $\|u_N\|_{H^k}\sim N^{k-\frac12}$. 
Similarly, define $v_N:=\chi_{B}+\chi_{-B}$ for the rectangle $B$ where
$$ 2N+1-2N^{-1}\leq \xi_1 \leq 2N+1+2N^{-1}\; \mbox{ and }\; -2\leq \xi_2\leq 2.$$
Note that $v_N$ is real-valued and $\|v_N\|_{H^\ell_x}\sim N^{\ell-\frac12}$.
We observe that
\begin{equation}\label{eq:lb_uv}
\widehat{u_N v_N}(\xi)\gs N^{-1},
\end{equation}
whenever $\xi=(\xi_1,\xi_2)$ satisfies
\begin{equation}\label{eq:reg_xi}
N+1-N^{-1} \leq \xi_1 \leq N+1+N^{-1} \;\mbox{ and }\; -1\leq \xi_2\leq 1.
\end{equation}
We write $2\Re (e^{-it'\langle \nabla \rangle}v_N)=(e^{-it'\langle \nabla \rangle}+e^{it'\langle \nabla \rangle})v_N$.
For $\xi$ satisfying \eref{eq:reg_xi} and $N^{-1}\ll |t|\ll T$ it holds
\begin{eqnarray*}
\fl &\left|\mathcal{F}_x\left(\int_0^te^{i(t-t')\Delta}\left(e^{it'\Delta}u_N (e^{-it'\langle \nabla \rangle}+e^{it'\langle \nabla \rangle})v_N\right)dt'\right)(\xi)\right|\\
\fl =&\left|\int \int_0^t e^{it'(|\xi|^2-|\eta|^2)}(e^{-it'\langle \xi-\eta\rangle}+e^{it'\langle \xi-\eta \rangle}) dt'
\widehat{u_N}(\eta)\widehat{v_N}(\xi-\eta)d\eta\right|
\gs |t|N^{-1}
\end{eqnarray*}
by \eref{eq:lb_uv} and because the first phase factor $|\xi|^2-|\eta|^2-\langle \xi-\eta\rangle $ is bounded whenever $\eta\in A$ and \eref{eq:reg_xi} holds for $\xi$, whereas the second phase factor $|\xi|^2-|\eta|^2+\langle \xi-\eta\rangle$ is of size $N$ in this region.

Integrating over this region \eref{eq:reg_xi} gives
\begin{eqnarray*}
\left\|\int_0^te^{i(t-t')\Delta}\left(e^{it'\Delta}u_N \Re\left(e^{-it'\langle \nabla \rangle} v_N\right)\right)dt'\right\|_{H^{k}_x}
\gs |t| N^{k-\frac32},
\end{eqnarray*}
which implies the claim.
\end{proof}

The following proposition is based on a variant of the example from the proof of Proposition \ref{prop:c2}.
\begin{proposition}\label{prop:unbounded_der2} Fix $0<T\leq 1$.
For all $N\gg 1$ there exists $u_N\in H^k_x$ such that
\begin{equation*}
\fl \sup_{|t|\leq T}\left\|\int_0^te^{-i(t-t')\langle \nabla \rangle}\frac{\Delta}{\langle \nabla\rangle}\left(e^{it'\Delta}u_N \overline{e^{it'\Delta}u_N}\right)dt'\right\|_{H^{\ell}_x}\gs N^{\ell-2k+\frac12}\|u_N\|_{H^k_x}^2\; ,
\end{equation*}
where the constant is independent of $N$.
\end{proposition}
\begin{proof}
Set $\hat u_N:=\chi_{D_1}+\chi_{D_2}$, where $D_1$ is the rectangle
where $\xi=(\xi_1,\xi_2)$ satisfies
$$ N+1-N^{-1}\leq \xi_{1}\leq N+1+N^{-1} \;\mbox{ and }\; -1\leq \xi_2\leq 1,$$
and $D_2$ is the rectangle where
$$ -N-2N^{-2}\leq \xi_1 \leq -N+2N^{-1}\;\mbox{ and }\; -2\leq \xi_2\leq 2.$$
Then, $\|u_N\|_{H^k}\sim N^{k-\frac12}$. We observe that
\begin{equation}\label{eq:lb_u}
\widehat{u_N \bar u_N}(\xi)\gs N^{-1},
\end{equation}
whenever $\xi=(\xi_1,\xi_2)$ satisfies
\begin{equation}\label{eq:region_xi}
2N+1-N^{-1} \leq \xi_1 \leq 2N+1+N^{-1} \;\mbox{ and }\; -1\leq \xi_2\leq 1.
\end{equation}
Therefore, for such $\xi$ and $|t|\ll 1$ it holds
\begin{eqnarray*}
 &\left|\mathcal{F}_x\left(\int_0^te^{i(t-t')\langle \nabla \rangle}\frac{\Delta}{\langle \nabla\rangle}\left(e^{it'\Delta}u_N \overline{e^{it'\Delta}u_N}\right)dt'\right)(\xi)\right|\\
\sim {}&|\xi|\left|\int_0^t\int e^{it'(\langle \xi\rangle-|\eta|^2+|\xi-\eta|^2)}\widehat{u_N}(\eta)\widehat{\overline{u_N}}(\xi-\eta)d\eta dt'\right|
\gs |t|
\end{eqnarray*}
by \eref{eq:lb_u}, $|\xi|\sim N$ and because the phase factor $\langle \xi \rangle-|\eta|^2+|\xi-\eta|^2$ is bounded whenever \eref{eq:region_xi} holds. Integrating over this region \eref{eq:region_xi} gives
\begin{eqnarray*}
\left\|\int_0^te^{-i(t-t')\langle \nabla \rangle}\frac{\Delta}{\langle \nabla\rangle}
 \left(e^{it\Delta}u_N \overline{e^{it\Delta}u_N}\right)dt'\right\|_{H^{\ell}_x}
\gs |t| N^{\ell-\frac12}
\end{eqnarray*}
and the claim follows.
\end{proof}
Finally, we indicate how we use Propositions \ref{prop:unbounded_der1} and \ref{prop:unbounded_der2} to prove Theorem \ref{thm:sharp}.
\begin{proof}[Proof of Theorem \ref{thm:sharp}]
Proposition \ref{prop:unbounded_der1} shows that for $\ell<-\frac12$ the first component of 
the directional (Fr\'echet) derivative of second order of the flow map to the reduced system \eref{eq:zs-red} at $0$ with respect to the direction $(u_0,v_0)=(u_N,v_N)$ is unbounded.

Proposition \ref{prop:unbounded_der2} shows that for $\ell-2k+\frac12>0$ the second component of
the directional derivative of second order of the flow map to the reduced system \eref{eq:zs-red} at $0$ with respect to the direction $(u_0,v_0)=(u_N,0)$ is unbounded.

If the flow map to the original system \eref{eq:zs} were $C^2$ then
we could conclude that the flow map for the reduced system is $C^2$ by
the arguments in Section \ref{sect:reduced}.  But this contradicts to
the assertions above.
\end{proof}

\appendix

\section{Alternative proof of Proposition \ref{prop:trans_low_mod}}
\label{sect:alt_proof}
\setcounter{section}{1}
\noindent
Here we present an alternate proof of Proposition \ref{prop:trans_low_mod}
that does not make use of the restriction theorem from
\cite{bejenaru_convolution_2008}.  The main source of technique for
the proof that follows is Colliander--Delort--Kenig--Staffilani
\cite{colliander_bilinear_2001}.

\begin{proof} We abuse notation and replace $g_2$ by $g_2(-\, \cdot)$
  and change variables $\zeta_2\mapsto -\zeta_2$ to obtain the usual
  convolution structure. From now on it holds $|\tau_2-|\xi_2|^2|\sim
  L_2$ within the support of $g_2$.

By the  change of variables $\tau_1=-|\xi_1|^2 + c_1$,
$\tau_2=|\xi_2|^2 +c_2$ and by applying the Cauchy-Schwarz inequality
with respect to $c_1$ and $c_2$ it suffices to consider the trilinear
expression
\[\fl
T(g_{1,c_1},g_{2,c_2},f) = 
 \int  g_{1,c_1}(\xi_1) g_{2,c_2} (\xi_2)
 f(\xi_1+\xi_2,|\xi_2|^2-|\xi_1|^2+c_1+c_2) d\xi_1 d\xi_2 
\]
where $g_{k,c_k} (\xi) = g_{k}(\xi,(-1)^k|\xi|^2+c_k)$ for $k=1,2$ and
$f$ is localized in the region $|\tau-|\xi|| \leq L$, and prove that
\begin{equation} \label{eq:reduce2}
|T(g_{1,c_1},g_{2,c_2},f)|
\ls \frac{A^{1/2}L^{1/2}}{N_1} \| g_{1,c_1} \|_{L^2_\xi} \| g_{2,c_2} \|_{L^2_\xi} \| f \|_{L^2}.
\end{equation}

We exploit the geometry of the problem in order to better localize the
interacting elements.  Taking into account the angular localization
and separation of $\xi_1$ and $\xi_2$ which is $\sim A^{-1}$ and their
size localization, it follows that after a rotation we may assume that
$\xi_{1,1}>0$, $\xi_{1,2}>0$ with $\xi_{1,1}\sim N_1$ and
$\xi_{1,2}\sim N_1A^{-1}$, and that either Case 1 or Case 2 below
holds (see Figure \ref{F:cases}).
\begin{figure}
\centering
\begin{picture}(0,0)%
\epsfig{file=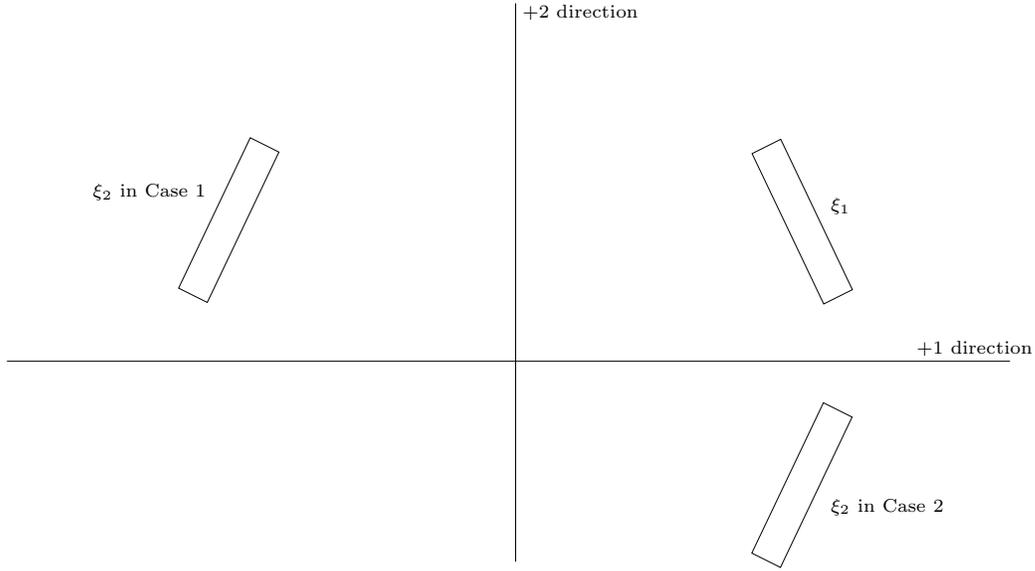}%
\end{picture}%
\setlength{\unitlength}{2368sp}%
\begingroup\makeatletter\ifx\SetFigFont\undefined%
\gdef\SetFigFont#1#2#3#4#5{%
  \reset@font\fontsize{#1}{#2pt}%
  \fontfamily{#3}\fontseries{#4}\fontshape{#5}%
  \selectfont}%
\fi\endgroup%
\begin{picture}(10644,5934)(1264,-5908)
\put(2176,-2011){\makebox(0,0)[lb]{\smash{{\SetFigFont{7}{8.4}{\rmdefault}{\mddefault}{\updefault}{\color[rgb]{0,0,0}$\xi_2$ in Case 1}%
}}}}
\put(9901,-5311){\makebox(0,0)[lb]{\smash{{\SetFigFont{7}{8.4}{\rmdefault}{\mddefault}{\updefault}{\color[rgb]{0,0,0}$\xi_2$ in Case 2}%
}}}}
\put(9901,-2161){\makebox(0,0)[lb]{\smash{{\SetFigFont{7}{8.4}{\rmdefault}{\mddefault}{\updefault}{\color[rgb]{0,0,0}$\xi_1$}%
}}}}
\put(10801,-3661){\makebox(0,0)[lb]{\smash{{\SetFigFont{7}{8.4}{\rmdefault}{\mddefault}{\updefault}{\color[rgb]{0,0,0}$+1$ direction}%
}}}}
\put(6676,-136){\makebox(0,0)[lb]{\smash{{\SetFigFont{7}{8.4}{\rmdefault}{\mddefault}{\updefault}{\color[rgb]{0,0,0}$+2$ direction}%
}}}}
\end{picture}%
\caption{After rotation, the two possible positions of $\xi_2$ in
    the proof of Prop. \ref{prop:trans_low_mod}, labeled as Case 1 and
    Case 2.}
  \label{F:cases}
\end{figure}

\noindent\textit{Case 1}.  $\xi_{2,1}<0$, $\xi_{2,2}>0$ with $|\xi_{2,1}|\sim N_1$ and $|\xi_{2,2}|\sim N_1A^{-1}$.

\noindent\textit{Case 2}.  $\xi_{2,1}>0$, $\xi_{2,2}<0$ with $|\xi_{2,1}|\sim N_1$ and $|\xi_{2,2}|\sim N_1A^{-1}$.
\begin{figure}
\centering
\begin{picture}(0,0)%
\epsfig{file=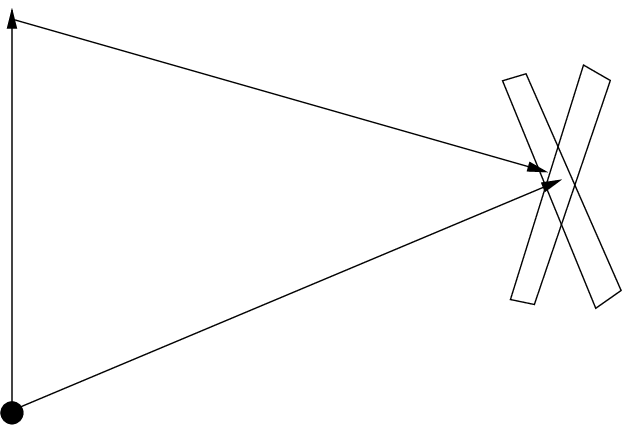}%
\end{picture}%
\setlength{\unitlength}{2763sp}%
\begingroup\makeatletter\ifx\SetFigFont\undefined%
\gdef\SetFigFont#1#2#3#4#5{%
  \reset@font\fontsize{#1}{#2pt}%
  \fontfamily{#3}\fontseries{#4}\fontshape{#5}%
  \selectfont}%
\fi\endgroup%
\begin{picture}(4272,2870)(3293,-4794)
\put(5551,-4186){\makebox(0,0)[lb]{\smash{{\SetFigFont{8}{9.6}{\rmdefault}{\mddefault}{\updefault}{\color[rgb]{0,0,0}support of $\xi-\xi_2$}%
}}}}
\put(3451,-3211){\makebox(0,0)[lb]{\smash{{\SetFigFont{8}{9.6}{\rmdefault}{\mddefault}{\updefault}{\color[rgb]{0,0,0}$\xi=\xi_1+\xi_2$}%
}}}}
\put(4351,-2161){\makebox(0,0)[lb]{\smash{{\SetFigFont{8}{9.6}{\rmdefault}{\mddefault}{\updefault}{\color[rgb]{0,0,0}$-\xi_2$}%
}}}}
\put(4126,-4636){\makebox(0,0)[lb]{\smash{{\SetFigFont{8}{9.6}{\rmdefault}{\mddefault}{\updefault}{\color[rgb]{0,0,0}$\xi_1$}%
}}}}
\put(5626,-2311){\makebox(0,0)[lb]{\smash{{\SetFigFont{8}{9.6}{\rmdefault}{\mddefault}{\updefault}{\color[rgb]{0,0,0}support of $\xi_1$}%
}}}}
\end{picture}

\begin{picture}(0,0)%
\epsfig{file=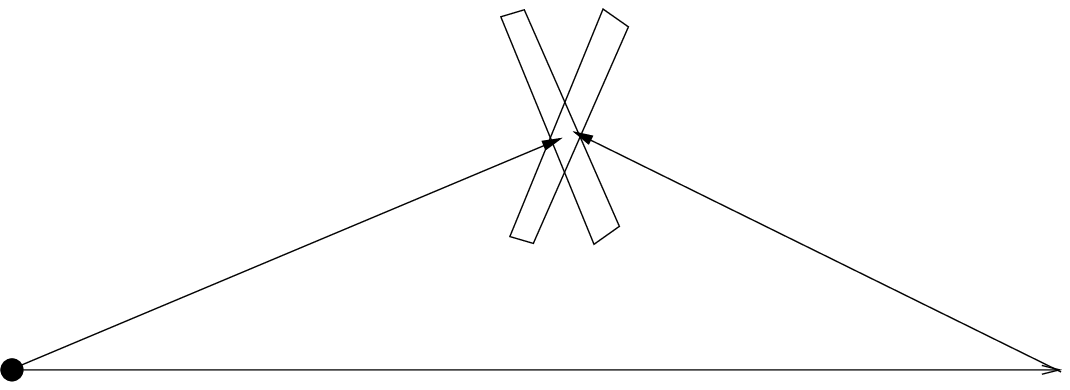}%
\end{picture}%
\setlength{\unitlength}{2763sp}%
\begingroup\makeatletter\ifx\SetFigFont\undefined%
\gdef\SetFigFont#1#2#3#4#5{%
  \reset@font\fontsize{#1}{#2pt}%
  \fontfamily{#3}\fontseries{#4}\fontshape{#5}%
  \selectfont}%
\fi\endgroup%
\begin{picture}(7295,2769)(1343,-5819)
\put(6676,-4561){\makebox(0,0)[lb]{\smash{{\SetFigFont{8}{9.6}{\rmdefault}{\mddefault}{\updefault}{\color[rgb]{0,0,0}$-\xi_2$}%
}}}}
\put(4201,-5761){\makebox(0,0)[lb]{\smash{{\SetFigFont{8}{9.6}{\rmdefault}{\mddefault}{\updefault}{\color[rgb]{0,0,0}$\xi=\xi_1+\xi_2$}%
}}}}
\put(3076,-4561){\makebox(0,0)[lb]{\smash{{\SetFigFont{8}{9.6}{\rmdefault}{\mddefault}{\updefault}{\color[rgb]{0,0,0}$\xi_1$}%
}}}}
\put(3526,-3511){\makebox(0,0)[lb]{\smash{{\SetFigFont{8}{9.6}{\rmdefault}{\mddefault}{\updefault}{\color[rgb]{0,0,0}support of $\xi_1$}%
}}}}
\put(5626,-3586){\makebox(0,0)[lb]{\smash{{\SetFigFont{8}{9.6}{\rmdefault}{\mddefault}{\updefault}{\color[rgb]{0,0,0}support of $\xi-\xi_2$}%
}}}}
\end{picture}%
  \caption{Depiction of Case 1 (top) and Case 2 (bottom). $\xi_1$ is supported in an annular
    ring $D_1$ of thickness $L/N_1$ and $\xi_2$ is supported in an
    annular ring $D_2$ of thickness $L/N_1$.  For fixed
    $\xi=\xi_1+\xi_2$, $\xi_1$ is confined to $D_1$ and also to
    $\xi-D_2$. These two sets have thickness $L/N_1$ but also meet at
    an angle $A^{-1}$, and thus $\xi_{1,1}$ is confined to an interval
    of size $L/N_1$ and $\xi_{1,2}$ is confined to an
    interval of size $LA/N_1$}
  \label{F:case12}
\end{figure}

\noindent In addition we consider the following  two cases separately, Case A: $L\geq N$ and Case B: $L\leq N$.

\smallskip

\noindent \textit{Case A}.  Suppose that $L\geq N$.  Since $|\tau-|\xi||\leq L$, we have that $|\xi_2|^2-|\xi_1|^2$ is confined to an interval of size $L$, and thus $|\xi_2|-|\xi_1|$ is confined to an interval of size $L/N_1$.   By the ``orthogonality'' Lemma \ref{L:localize} below, and Cauchy-Schwarz, we might as well assume that $|\xi_1|$ and $|\xi_2|$ are confined to fixed intervals of size $L/N_1$. Note that in the two cases outlined above, we have (see Fig. \ref{F:case12})

\smallskip

\noindent\textit{Case A1}.  $\xi_{1,2}+\xi_{2,2} \sim N_1A^{-1}$ and if $\xi=\xi_1+\xi_2$ is fixed, then $\xi_{1,1}$ is contained in an interval of size $LN_1^{-1}$.

\smallskip

\noindent\textit{Case A2}.  $\xi_{1,1}+\xi_{2,1} \sim N_1$ and if $\xi=\xi_1+\xi_2$ is fixed, then $\xi_{1,2}$ is contained in an interval of size $LAN_1^{-1}$.

\smallskip

\noindent Let $\mu =\xi_1+\xi_2$, $\nu =-|\xi_1|^2+|\xi_2|^2+c_1+c_2$, and in Case 1 let $\sigma = \xi_{1,1}$, but in Case 2 let $\sigma =\xi_{1,2}$.  Denote by $J$ the Jacobian determinant. We have
\[
J= \begin{array}{ c | c  c  c  c |}
& \xi_{1,1} & \xi_{1,2} & \xi_{2,1} & \xi_{2,2} \\
\hline 
\mu_1 & 1 & 0 & 1 & 0 \\
\mu_2 & 0 & 1 & 0 & 1 \\
\nu   &  -2\xi_{1,1}  &  -2\xi_{1,2} & 2\xi_{2,1} & 2\xi_{2,2} \\
\sigma & * & * & 0 & 0
\end{array}
\]
and thus
\[
|J| =
\cases{
2|\xi_{2,2}+\xi_{1,2}| & in Case 1\\
2|\xi_{2,1}+\xi_{1,1}| & in Case 2\\}
 \sim
\cases{
N_1A^{-1} & in Case 1\\
N_1 & in Case 2\\}\;.
\]
So, $|J|$ is essentially constant over the region of integration, and can be removed from the integration.  We obtain
\begin{eqnarray*}
\fl T(g_{1,c_1},g_{2,c_2},f)
=\int   g_{1,c_1}(\xi_1) g_{2,c_2} (\xi_2) f(\mu ,\nu ) |J|^{-1} \, d\mu \, d\nu \, d\sigma 
\leq |J|^{-1/2} I_1 I_2  
\end{eqnarray*}
where
\[\fl
I_1 =\left( \int_{\mu,\nu,\sigma} |J|^{-1} |g_{1,c_1}(\xi_1) g_{2,c_2} (\xi_2)|^2 \, d\mu\, d\nu\, d\sigma\right)^{1/2} = \|g_{1,c_1}\|_{L^2_{\xi_1}} \|g_{2,c_2}\|_{L^2_{\xi_2}}
\]
and
\[
I_2 =\left( \int_{\mu,\nu}  |f(\mu ,\nu)|^2 \left( \int_\sigma d\sigma \right) \, d\mu \, d\nu\, \right)^{1/2} \,.
\]
The measure of the support of $\sigma$, for fixed $\mu=\xi_1+\xi_2$, in Case 1 is $LN_1^{-1}$ and in Case 2 is $LAN_1^{-1}$.  Thus, we obtain \eref{eq:reduce2}.

\smallskip

\noindent\textit{Case B}.  Now suppose that $ L\leq N$. 
Let $\{E_j\}$ be a partition of $[0,+\infty)$ into intervals of length $L$.  Then the left side of \eref{eq:reduce2} becomes
\[
\sum_j \int g_1(\xi_1)g_2(\xi_2) f(\xi_1+\xi_2, \cdot) \chi_{E_j}(|\xi_1+\xi_2|) \, d\xi_1\, d\xi_2.
\]
For a fixed $j$, we have that $|\xi|$ is localized to an interval of length $L$, and  since $|\tau-|\xi||\leq L$, we obtain that $|\xi_2|^2-|\xi_1|^2$ is localized to an interval of size $L$, from which it follows that $|\xi_2|-|\xi_1|$ is localized to an interval of length $L/N_1$.  We can now follow the argument of Case A to obtain the bound
\begin{eqnarray*}
\fl |T(g_{1,c_1},g_{2,c_2},f)|
\ls \frac{A^{1/2}L^{1/2}}{N_1} 
\sum_j  \|g_1(\xi_1)g_2(\xi_2)
\chi_{E_j}(|\xi_1+\xi_2|)\|_{L^2_{\xi_1\xi_2}} \|f(\xi,\tau) \chi_{E_j}(|\xi|) \|_{L^2_{\xi \tau}} \,.
\end{eqnarray*}
Applying Cauchy-Schwarz with respect to  $j$ we complete the proof of \eref{eq:reduce2}.
\end{proof}

\begin{lemma}\label{L:localize}
Suppose $N_1\gtrsim 1$, $1\lesssim A \ll N_1$, $k\ll N_1^2$ and that $x,y\geq 0$ satisfy 
\[k \leq x^2-y^2 \leq k + N_1A^{-1}\,, \qquad  \frac{1}{4} N_1 \leq x,y \leq 4N_1 \, .\]  
Decompose $[\frac{1}{4} N_1,4N_1]$ into a sequence of intervals $\{I_j\}$ each of length $A^{-1}$.  Then there is a mapping $j\mapsto k(j)$ such that
\[ 
y\in I_j \; \Rightarrow \;x\in I_{k(j)-100}\cup \cdots \cup I_{k(j)+100} \,.
\]  
Moreover, as $j$ ranges over the full set of intervals, $k(j)$ hits a particular element no more than $100$ times.
\end{lemma}
\begin{proof}
We take $I_j = [A^{-1}(j-\frac12), A^{-1}(j+\frac12))]$ (so $j$ ranges from $AN_1/4$ to $4AN_1$).  Suppose that $y\in I_j$.  Then $|y-A^{-1}j| \leq A^{-1}$, and therefore
\[k-4N_1A^{-1} \leq x^2 - A^{-2}j^2 \leq k + 4N_1A^{-1},\]
which implies that
\[(A^{-2}j^2 + k - 4N_1A^{-1})^{-1/2} \leq x \leq (A^{-2}j^2 + k + 4N_1A^{-1})^{-1/2}\,.\]
The length of this interval is 
\[\frac{8N_1A^{-1}}{
(A^{-2}j^2 + k - 4N_1A^{-1})^{-1/2}+ (A^{-2}j^2 + k + 4N_1A^{-1})^{-1/2}} \lesssim A^{-1}.\]
Also, as we increment from $j$ to $j+1$, the left endpoint of the interval advances by an amount
\[\frac{2A^{-2}j}{(A^{-2}j^2 + k - 4N_1A^{-1})^{-1/2}+ (A^{-2}j^2 + k + 4N_1A^{-1})^{-1/2}} \gtrsim A^{-1},\]
and the claim follows.\end{proof}

\bibliographystyle{plain} \bibliography{zs-refs}

\end{document}